\documentclass[a4paper,11pt]{article}
\usepackage{amsthm,amsmath}
\usepackage[top=2.5cm, bottom=2.5cm, left=2cm, right=2cm]{geometry}
\usepackage[utf8]{inputenc} 
\usepackage[english]{babel}
\usepackage[T1]{fontenc}   
\usepackage{graphicx}
\usepackage{comment}
\usepackage{amsthm,amsmath,amsfonts,stmaryrd,mathrsfs,amssymb}
\usepackage{mathtools}
\usepackage[inline]{enumitem}
\usepackage{chemarrow}
\usepackage{tikz}
\usepackage{subfig}
\usepackage{multirow}
\usepackage{array}
\usepackage{bigints}
\usepackage{etoolbox}
\usepackage{color}
\usepackage[unicode,colorlinks=true,linkcolor=blue,citecolor=red,pdfencoding=auto,psdextra]{hyperref}
\usepackage[backend=biber,style=trad-plain,abbreviate=false,doi=false,isbn=false,url=false]{biblatex}
\DeclareSourcemap{
	\maps[datatype=bibtex]{
		\map[overwrite=true]{
			\step[fieldsource=fjournal]
			\step[fieldset=journal, origfieldval]
		}
	}
}

\renewbibmacro*{journal}{%
	\printfield{journaltitle}%
	\setunit{\subtitlepunct}%
	\printfield{journalsubtitle}}
\setenumerate[1]{label=(\roman*), font = \normalfont}
\newcommand{\ensemblenombre}[1]{\mathbb{#1}}
\newcommand{\N}{\ensemblenombre{N}}

\newcommand{\R}{\ensemblenombre{R}}

\newcommand{\D}{\mathbb{D}}
\newcommand{\K}{\mathbb{K}}
\newcommand{\Ec}[1]{\mathbb{E} \left[#1\right]}

\newcommand{\Pp}[1]{\mathbb{P} \left(#1\right)}

\newcommand{\Ecsq}[2]{\mathbb{E} \left[#1\mathrel{}\middle|\mathrel{}#2\right]}

\newcommand{\Ppsq}[2]{\mathbb{P} \left(#1\mathrel{}\middle|\mathrel{}#2\right)}

\newcommand{\intervalle}[4]{\mathopen{#1}#2
                            \mathclose{}\mathpunct{},#3
                            \mathclose{#4}}
\newcommand{\intervalleff}[2]{\intervalle{[}{#1}{#2}{]}}
\newcommand{\intervalleof}[2]{\intervalle{(}{#1}{#2}{]}}
\newcommand{\intervallefo}[2]{\intervalle{[}{#1}{#2}{)}}
\newcommand{\intervalleoo}[2]{\intervalle{(}{#1}{#2}{)}}
\newcommand{\intervalleentier}[2]{\intervalle\llbracket{#1}{#2}
                               \rrbracket}
       
\newcommand{\petito}[1]{o\mathopen{}\left(#1\right)}
\newcommand{\petitoom}[1]{o^\omega\mathopen{}\left(#1\right)}
\newcommand{\petiton}{o_n\mathopen{}\left(1\right)}
\newcommand{\petitoe}{o_\epsilon\mathopen{}\left(1\right)}

\newcommand{\grandO}[1]{O\mathopen{}\left(#1\right)}
\newcommand{\tend}[2]{\underset{#1}{\overset{#2}{\longrightarrow}}}
\newcommand{\enstq}[2]{\left\lbrace#1\mathrel{}\middle|\mathrel{}#2\right\rbrace}

\newcommand{\ind}[1]{\mathbf{1}_{\left\lbrace #1 \right\rbrace}}  
\newcommand{\cT}{\mathcal{T}}
\newcommand{\cB}{\mathcal{B}}

\newcommand{\cL}{\mathcal{L}}

\newcommand{\cP}{\mathcal{P}}

\newcommand{\cH}{\mathcal{H}}
\newcommand{\cF}{\mathcal{F}}
\newcommand{\cE}{\mathcal{E}}

\newcommand{\bb}{\mathbf{b}}
\newcommand{\brho}{\boldsymbol\rho}
\newcommand{\bd}{\mathbf{d}}
\newcommand{\bnu}{\boldsymbol\nu}

\newcommand{\bB}{\mathsf{B}}
\newcommand{\bRho}{\rho}
\newcommand{\bD}{\mathsf{D}}
\newcommand{\bNu}{\nu}

\newcommand{\bM}{\mathsf{M}}

\newcommand{\parametre}{\gamma}

\newcommand{\vM}{\mathtt{M}}

\DeclareMathOperator{\haut}{ht}
\DeclareMathOperator{\diam}{diam}

\DeclareMathOperator{\dist}{d}
\DeclareMathOperator{\Vol}{Vol}
\DeclareMathOperator{\Ball}{B}
\DeclareMathOperator{\supp}{supp}

\newcommand{\abs}[1]{\left\lvert#1\right\rvert}

\newtheorem{theorem}{Theorem}

\newtheorem{proposition}[theorem]{Proposition}

\newtheorem{lemma}[theorem]{Lemma}
\newtheorem{definition-proposition}[theorem]{Definition-Proposition}

\newtheorem{remark}[theorem]{Remark}

\newtheorem{hypothesismx}{Hypothesis}
\newenvironment{hypothesis}[1]
{\renewcommand\thehypothesismx{#1}\hypothesismx}
{\endhypothesismx}

\begin{document}
	\title{Random gluing of metric spaces}
	\author{Delphin Sénizergues}
	\maketitle
	\abstract{
	We construct random metric spaces by gluing together an infinite sequence of pointed metric spaces that we call \emph{blocks}. At each step, we glue the next block to the structure constructed so far by randomly choosing a point on the structure and then identifying it with the distinguished point of the block. The random object that we study is the completion of the structure that we obtain after an infinite number of steps. 
	In \cite{curien_random_2017}, Curien and Haas study the case of segments, where the sequence of lengths is deterministic and typically behaves like $n^{-\alpha}$. They proved that for $\alpha>0$, the resulting tree is compact and that the Hausdorff dimension of its set of leaves is $\alpha^{-1}$.
	The aim of this paper is to handle a much more general case in which the blocks are i.i.d.\ copies of the same random metric space, scaled by deterministic factors that we call $(\lambda_n)_{n\geq 1}$. We work under some conditions on the distribution of the blocks ensuring that their Hausdorff dimension is almost surely $d$, for some $d\geq0$. 
	We also introduce a sequence $(w_n)_{n\geq1}$ that we call the \emph{weights} of the blocks. At each step, the probability that the next block is glued onto any of the preceding blocks is proportional to its weight.
	The main contribution of this paper is the computation of the Hausdorff dimension of the set $\cL$ of points which appear during the completion procedure when the sequences $(\lambda_n)_{n\geq 1}$ and $(w_n)_{n\geq1}$ typically behave like a power of $n$, say $n^{-\alpha}$ for the scaling factors and $n^{-\beta}$ for the weights, with $\alpha>0$ and $\beta\in\R$. For a large domain of $\alpha$ and $\beta$ we have the same behaviour as the one observed in \cite{curien_random_2017}, which is that $\mathrm{dim_H}(\cL)=\alpha^{-1}$. However for $\beta>1$ and $\alpha<1/d$, our results reveal an interesting phenomenon: the dimension has a non-trivial dependence in $\alpha$, $\beta$ and $d$, namely \[\displaystyle\mathrm{dim_H}(\cL)=\frac{2\beta-1-2\sqrt{(\beta-1)(\beta-\alpha d)}}{\alpha}.\] The computation of the dimension in the latter case involves new tools, which are specific to our model.}
\begin{figure}
	\begin{center}
		\includegraphics[scale=0.20]{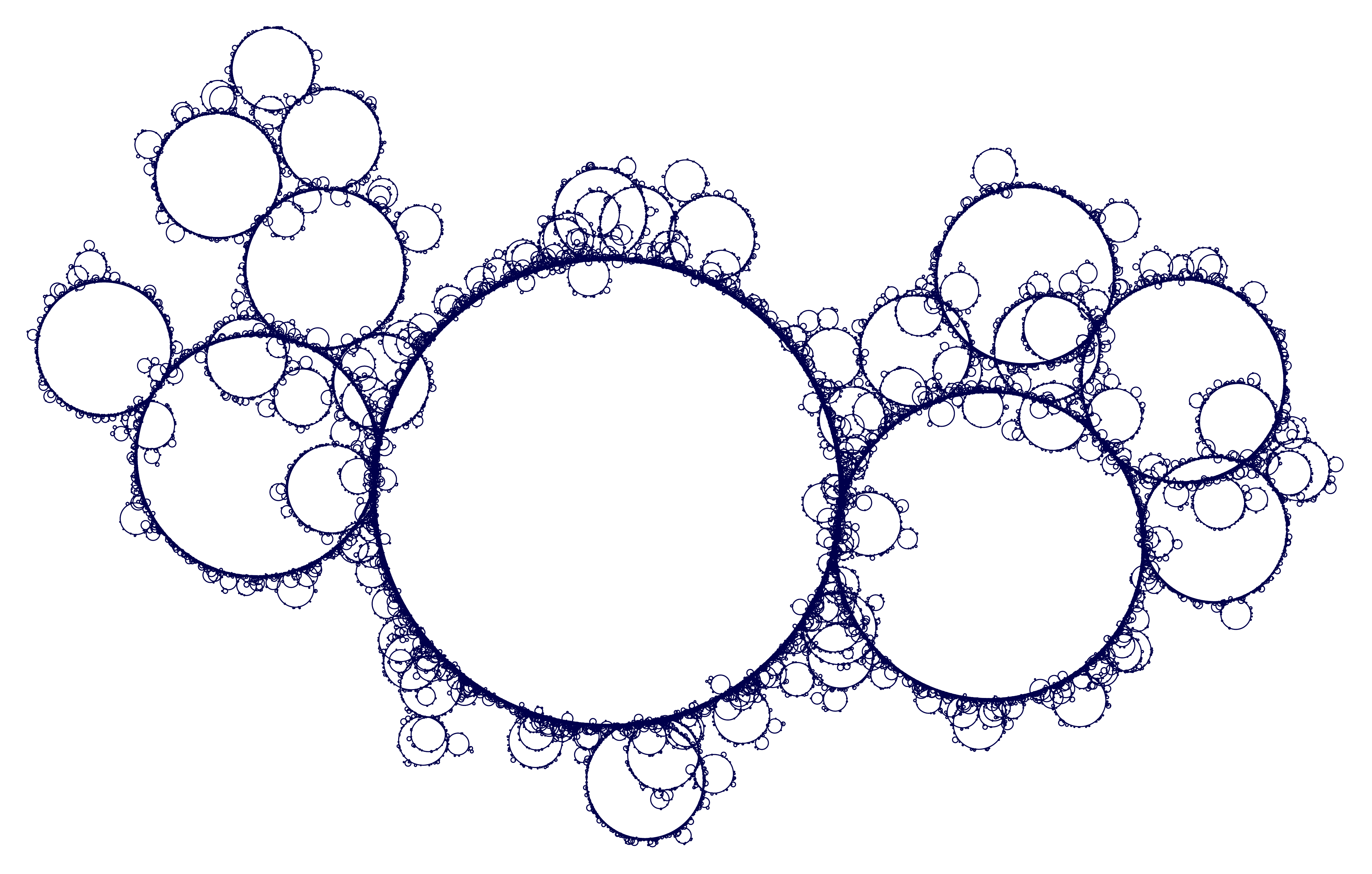}
		\caption{Gluing of circles of radii $\lambda_n=n^{-3/5}$, with weights $w_n=n^{-3/2}$. The Hausdorff dimension of the resulting metric space is $(\frac{10}{3}-\sqrt{5})$.}\label{gluing:cercles}
	\end{center}
\end{figure}

\section{Introduction}

Let us recall Aldous' famous line-breaking construction of the Brownian CRT (Continuum Random Tree) in \cite{aldous_continuum_1991}. On the half-line $\intervallefo{0}{\infty}$, consider $C_1,\ C_2, \dots ,C_n$ the points of a Poisson process with intensity $t\mathrm{d}t$. Cut the half-line in closed intervals $\intervalleff{C_i}{C_{i+1}}$, which we call \emph{branches} (of length $C_{i+1}-C_i$). Starting from $\intervalleff{0}{C_1}$, construct a tree by recursively gluing the branch $\intervalleff{C_i}{C_{i+1}}$ to a random point chosen uniformly on the tree already constructed (i.e.\ under the normalised length measure). Aldous' Brownian CRT is the completion of the tree constructed after an infinite number of steps. 
This process can be generalised by using any arbitrary sequence $(\lambda_n)$ for the length of the successive branches. This model was introduced and studied by Curien and Haas in \cite{curien_random_2017}, who proved that when $\lambda_n=n^{-\alpha+\petito{1}}$ for some $\alpha>0$, the tree obtained is a.s.\ compact and has Hausdorff dimension $(1\vee\alpha^{-1})$. In \cite{amini_explosion_2017}, Amini et.\ al.\ obtained a necessary and sufficient condition on the sequence $(\lambda_n)$ for the almost sure compactness of the resulting tree, under the assumption that this sequence is non-increasing. In \cite{haas_asymptotics_2017}, Haas describes how the height of the tree explodes when $n\rightarrow \infty$ under the assumption that $\lambda_n\approx n^\alpha$, with $\alpha\geq 0$.

Our goal is to define a more general version of this model, in which the branches are replaced by arbitrary (and possibly random) measured metric spaces, and to investigate the compactness and the Hausdorff dimension of the resulting metric space. As we will see, in this broader context, a striking phenomenon (absent from \cite{curien_random_2017}) pops up. 
In all the paper we will work with
\[(\lambda_n)_{n\geq1} \quad \text{and} \quad (w_n)_{n\geq1},\] two sequences of non-negative real numbers that will be the \emph{scaling factors} and \emph{weights} of the metric spaces that we glue. All the scaling factors $(\lambda_n)_{n\geq1}$ are considered strictly positive, but the weights, except for the first one $w_1$, can possibly be null. 

\paragraph{Definition of the model and main results} Let us first present a simpler version of our construction, in which we construct a tree through an aggregation of segments.
For now the branches, which we denote $(\bb_n)_{n\geq 1}$, are segments of length $(\lambda_n)_{n\geq 1}$, rooted at one end and endowed with the Lebesgue measure normalised so that their respective total measure is $(w_n)_{n\geq1}$ (or endowed with the null measure for branches with vanishing weight). We then define a sequence $(\cT_n)_{n\geq 1}$ of increasing trees by gluing those branches as follows. First, $\cT_1=\bb_1$. Then, if $\cT_n$ is constructed, we build $\cT_{n+1}$ by first sampling a point $X_n$ chosen proportionally to the measure $\mu_n$ obtained by aggregating the measures concentrated on the branches $\bb_1, \dots, \bb_n$ and then gluing $\bb_{n+1}$ onto $\cT_n$ by identifying its root with $X_n$. Let $\cT^*$ be the increasing union of the trees $\cT_n$ for $n\geq 1$ and $\cT$ be the completion of $\cT^*$. Note that if $(w_n)=(\lambda_n)$, this model coincides with the one studied in \cite{curien_random_2017}.

We can compute the Hausdorff dimension of the resulting tree in the case where $(\lambda_n)$ and $(w_n)$ behave like powers of $n$, say $\lambda_n= n^{-\alpha}$ and $w_n= n^{-\beta}$. We define $\cL:=(\cT\setminus \cT^*)$ to which we refer as the set of \emph{leaves} of $\cT$. In this particular case it coincides, up to a countable set, with the set of points $x$ such that $\cT\setminus\{x\}$ remains connected. In the above context a trivial consequence of our main theorem is that $\cT$ is a.s.\ compact and
\[ \begin{aligned}
\mathrm{dim_H}(\cL)&= \frac{2\beta-1-2\sqrt{(\beta-1)(\beta-\alpha)}}{\alpha} \quad \text{if }\beta>1 \text{ and } \alpha<1,\\
&= \frac{1}{\alpha} \quad\text{otherwise,}\\
\end{aligned}\]
where $\mathrm{dim_H}(X)$ stands for the Hausdorff dimension of the metric space $X$, see Section~\ref{gluing:subsec:hausdorff dimension}.

Note that, since we can check that the dimension of the \emph{skeleton} $\cT^*$ is always $1$, we can recover the dimension of $\cT$ as $\mathrm{dim_H}(\cT)=\max(1, \mathrm{dim_H}(\cL))$. 
We see that $\mathrm{dim_H}(\cL)=\frac{1}{\alpha}$ as in \cite{curien_random_2017} for most values of $\beta$, however, a new phenomenon, absent from \cite{curien_random_2017}, happens in the case $\beta>1$ (the sum of the weights is finite) and $\alpha<1$ (the total length is infinite). In this case, the Hausdorff dimension of $\cT$ depends in a non-trivial manner on $\alpha$ and $\beta$.

\bigskip

Now we want to generalise it to sequences $(\bb_n)$ of more general metric spaces that we call \emph{blocks}, which can be random and possibly more elaborate than just segments.
Specifically, our blocks are based on the distribution of a random pointed measured compact metric space, $(\bB,\bD,\bRho,\bNu)$, with underlying set $\bB$, distance $\bD$, distinguished point $\bRho$ and endowed with a probability measure $\bNu$. We sometimes denote it $\bB$ by abuse of notation when no confusion is possible and we refer to it as the \emph{underlying random block}. We consider a sequence $\left((\bB_n,\bD_n,\bRho_n,\bNu_n)\right)_{n\geq 1}$ of i.i.d.\ random variables with the distribution of $(\bB,\bD,\bRho,\bNu)$ and define our blocks by setting
\begin{equation}\label{gluing:block scaling}
\forall n\geq 1, \quad (\bb_n,\bd_n,\brho_n,\bnu_n):= (\bB_n,\lambda_n\cdot\bD_n,\bRho_n,w_n\cdot \bNu_n),
\end{equation}
meaning that we dilate all the distances in the space $\bB_n$ by the factor $\lambda_n$ and scale the measure by $w_n$. 
We suppose that, $\lambda_n\approx n^{-\alpha}$ for some $\alpha>0$, and $w_n\approx n^{-\beta}$ for some $\beta\in\R$, in some loose sense which we make precise in the sequel. For technical reasons, we have to separate the case $\beta<1$, the case $\beta>1$ and $\beta=1$. This gives rise to the three hypotheses Hyp.~\ref{gluing:condition beta petit}, Hyp.~ \ref{gluing:condition beta grand} and Hyp.~\ref{gluing:condition beta 1}. 
For any $d\in\intervallefo{0}{\infty}$, we will introduce the Hypothesis~\ref{gluing:hypothese d} and suppose that the distribution of our underlying random block $(\bB,\bD,\bRho,\bNu)$ satisfies this hypothesis for some $d\geq0$. This hypothesis ensures that our random block exhibits a $d$-dimensional behaviour. We set out all these hypotheses just below the statement of our theorem. 

\textbf{Except in Section~\ref{gluing:subsec:construction}, we will always assume that the blocks are of the form \eqref{gluing:block scaling}. This is implicit in all our results. }

In this extended setting, we can perform the same gluing algorithm and build a sequence $(\cT_n)_{n\geq1}$ of random compact metric spaces by iteratively gluing the root of $\bb_{n+1}$ onto a point chosen in $\cT_{n}$ according to the measure $\mu_n$ obtained as the sum of the measures of the blocks $\bb_1, \dots, \bb_n$. Again $\cT^*=\bigcup_{n\geq1}\cT_n$ is called the \emph{skeleton} of the construction and its completion is still denoted $\cT$. See Figure \ref{gluing:cercles} for non-isometric, non-proper representation in the plane of a simulation of this model, with $\bB$ chosen to be almost surely a circle of unit length. As for the case of segments, we refer to $\cL= \left(\cT\setminus \cT^*\right)$ as the set of leaves of the construction. We can now state our main theorem.

\begin{theorem}\label{gluing:theoreme principal}
	Suppose that there exists $d\geq0$, such that $(\bB,\bD,\bRho,\bNu)$ satisfies Hypothesis~\ref{gluing:hypothese d}, and $\alpha>0$ and $\beta\in\R$ such that the sequences $(w_n)$ and $(\lambda_n)$ satisfy either Hyp.~\ref{gluing:condition beta grand} or Hyp.~\ref{gluing:condition beta petit}, or Hyp.~\ref{gluing:condition beta 1}.
	Then, almost surely, the structure $\cT$ resulting from the construction is compact, and
	\[ \begin{aligned}
	\mathrm{dim_H}(\cL)&= \frac{2\beta-1-2\sqrt{(\beta-1)(\beta-\alpha d)}}{\alpha}, \quad \text{if }\beta>1 \text{ and } \alpha<\frac{1}{d},\\
	&= \frac{1}{\alpha} \quad\text{otherwise.}
	\end{aligned}\]
\end{theorem}
\begin{figure}
	\centering
	\includegraphics[height=5cm]{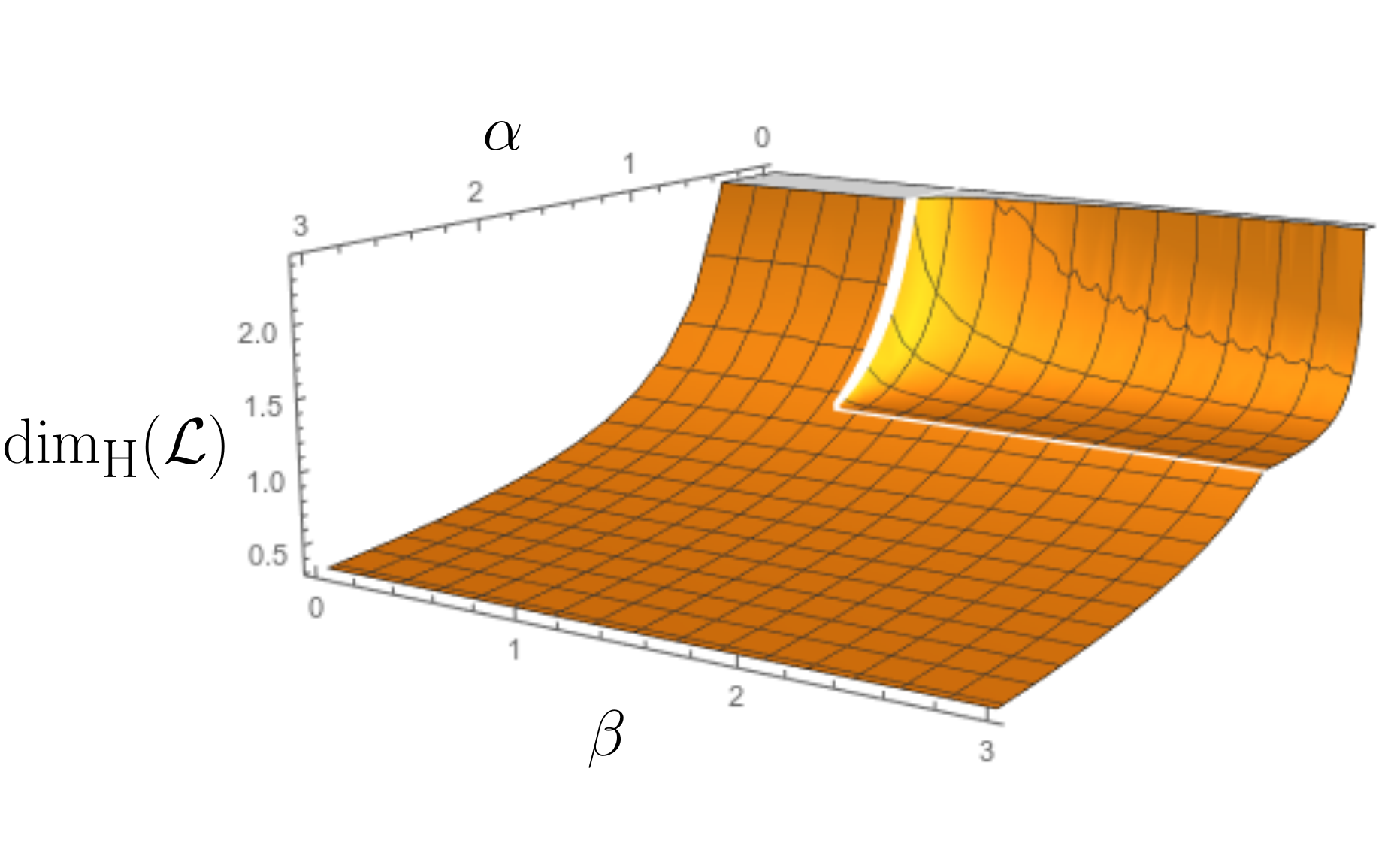}
	\caption{The plot represents the Hausdorff dimension of the leaves as a function of $\alpha$ and $\beta$, the dimension $d$ being fixed to $1$. The expression obtained for $\beta>1$ and $\alpha<\frac{1}{d}$ can be rewritten as $\displaystyle d+\frac{(\sqrt{\beta-\alpha d}-\sqrt{\beta-1})^2}{\alpha}$
		This expression is always larger than $d$ and smaller than $\frac{1}{\alpha}$,  and it is decreasing in $\alpha$ and $\beta$ on the domain on which we consider it. When $\beta\rightarrow 1$, it converges to the value $\frac{1}{\alpha}$ so that the function $(\alpha,\beta)\mapsto \dim_H(\cL)$ is continuous on the domain $\R^*_+\times \R$.
	}
\end{figure}

Remark that for $\beta>1$ the dimension of the set $\cL$ depends on the geometry of the underlying random block through $d$, its dimension. For $\beta\leq 1$, it is not the case, and actually the theorem remains true under much weaker hypotheses for the distribution of $(\bB,\bD,\bRho,\bNu)$, namely that $\bNu$ is not almost surely concentrated on $\{\bRho\}$, and that $\forall k\geq 0,\Ec{(\diam(\bB))^k}<\infty$, where $\diam(\cdot)$ denotes the diameter of a metric space. We could even replace the assumption that the blocks $\left((\bB_n,\bD_n,\bRho_n,\bNu_n)\right)_{n\geq 1}$ are i.i.d.\ by some weaker assumption but we do not do it for the sake of clarity. The proofs when $\beta\leq1$ are quite short and the interested reader can easily generalise them to a more general setting.  

\paragraph{Hypotheses of the theorem}Let us define and discuss the precise hypotheses of our theorem. First, let us describe the assumptions that we make on the sequences $(\lambda_n)$ and $(w_n)$.
We define \[ W_n=\sum_{k=1}^n w_k,\] and for all $\epsilon>0$, we set
\begin{equation}\label{gluing:def G epsilon} G^\epsilon:= \enstq{k\geq 1}{ w_k\geq k^{-\beta -\epsilon}, \ \lambda_k\geq k^{-\alpha-\epsilon}},
\end{equation}
and also
\begin{equation}\label{gluing:def G epsilon n} G_n^\epsilon:= \enstq{k\in\intervalleentier{n}{2n}}{w_k\geq n^{-\beta -\epsilon}, \ \lambda_k\geq n^{-\alpha-\epsilon}}.
\end{equation}
As said earlier, we separate the case $\beta<1$, the case $\beta>1$ and the case $\beta=1$.  
\begin{hypothesis}{$\bigcirc_{\alpha,\beta}$}\label{gluing:condition beta petit} We have $\alpha>0$ and $\beta<1$ and for all $n\geq1,\ \lambda_n\leq n^{-\alpha+\petito{1}}$ and $w_n\leq n^{-\beta+\petito{1}}$. Furthermore $W_n=n^{1-\beta+\petito{1}}$ and for all $\epsilon>0$,
	\[\liminf_{n\rightarrow\infty}\frac{\sum_{k=1}^{n}w_k \ind{k\in G^\epsilon}}{\sum_{k=1}^{n}w_k}>0.\]
\end{hypothesis}
The last display ensures\textbf{} that for all $\epsilon>0$, the set $G^\epsilon$ contains asymptotically a positive proportion of the total weight.

\begin{hypothesis}{$\diamond_{\alpha,\beta}$}\label{gluing:condition beta grand} We have $\alpha>0$ and $\beta>1$ and for all $n\geq1,\ \lambda_n\leq n^{-\alpha+\petito{1}}$ and $w_n\leq n^{-\beta+\petito{1}}$. Furthermore, for all $\epsilon>0$,
	\[\#G_n^\epsilon\underset{n\rightarrow\infty}{=} n^{1+\petito{1}}.\]
\end{hypothesis}
Under the stronger assumption $\lambda_n= n^{-\alpha+\petito{1}}$ and $w_n= n^{-\beta+\petito{1}}$, Hypothesis~\ref{gluing:condition beta grand} holds if $\beta> 1$ (resp.\ Hypothesis~\ref{gluing:condition beta petit}, if $\beta<1$). The case $\beta=1$ is slightly different and in this case we set
\begin{hypothesis}{$\square_{\alpha,1}$}\label{gluing:condition beta 1} We have $\alpha>0$ and $\beta=1$ and for all $n\geq1,\ \lambda_n\leq n^{-\alpha+\petito{1}}$ and $w_n\leq n^{-1+\petito{1}}$. Furthermore, for all $\epsilon>0$,
	\[\frac{1}{\log\log\log N} \sum_{k=N}^{N^{1+\epsilon}}\frac{w_k}{W_k} \ind{k\in G^\epsilon}\underset{N\rightarrow \infty}{\longrightarrow}+\infty.\]
\end{hypothesis}
Note that this last hypothesis requires in particular that $W_n\rightarrow\infty$ as $n\rightarrow\infty$. Now let us define Hypothesis~\ref{gluing:hypothese d}, for any $d\geq 0$, which will ensure that our random underlying block has the appropriate $d$-dimensional behaviour.
\begin{hypothesis}{$H_d$}\label{gluing:hypothese d}
	The law of the block $(\bB,\bD,\bRho,\bNu)$ satisfies the following conditions:
	\begin{enumerate}
		\item \label{gluing:dimension d}$\bullet$If $d=0$, the block $\bB$ is a finite metric space which is not a.s.\ reduced to a single point and such that the measure $\bNu$ satisfies $\bNu(\{x\})>0$, for all points $x\in\bB$.
		
		$\bullet$If $d> 0$, there exists an increasing function $\varphi:\intervalleff{0}{1}\rightarrow\intervalleff{0}{d/2}$, satisfying $\displaystyle\lim_{r\rightarrow 0}\varphi(r)=0$, such that almost surely, there exists a (random) $r_0\in\intervalleoo{0}{1}$ such that
		\begin{equation} \forall r\in\intervallefo{0}{r_0}, \ \forall x\in\bB,\  \quad r^{d+\varphi(r)}\leq \bNu \left( \Ball(x,r)\right)\leq r^{d-\varphi(r)}.
		\tag{$\star_{r_0}$} \label{gluing:controle boules}
		\end{equation}
		\item \label{gluing:nombre mininimal de boules} Let $N_r(\bB)$ be the minimal number of balls of radius $r$ needed to cover $\bB$. Then\[\Ec{N_r(\bB)}\leq r^{-d+\petito{1}}\quad \text{as} \quad r\rightarrow 0.\]
		\item \label{gluing:moment exponentiel 1}For all $k\geq 0$, we have  $\Ec{\diam(\bB)^k}<\infty$.
	\end{enumerate}
\end{hypothesis}
Here $\Ball(x,r)$ is the open ball centred at $x$ with radius $r$ and the notation $\diam(\bB)$ denotes the diameter of $\bB$, defined as the maximal distance between two points of $\bB$. The conditions \ref{gluing:dimension d} and \ref{gluing:nombre mininimal de boules} ensure that the blocks that we glue together have dimension $d$. The condition \ref{gluing:moment exponentiel 1} ensures that the blocks cannot be too big. In the paper, some results are stated under some weaker assumptions on the distribution of random block $(\bB,\bD,\bRho,\bNu)$ and they are hence all still valid under Hypothesis~\ref{gluing:hypothese d}. 

\paragraph{Motivations}
The assumptions of Theorem~\ref{gluing:theoreme principal} are rather general and various known models fall into our setting. First, let us cite two constructions that were already covered by the work presented in \cite{curien_random_2017}. Of course we have Aldous' line-breaking construction of the CRT but let us also cite the work of Ross and Wen in \cite{ross_scaling_2018}, 
in which the authors study a discrete model of growing trees and prove that its scaling limit can be described as a line-breaking procedure à la Aldous using a Poisson process of intensity $t^l\mathrm{d}t$, with $l$ an integer. The Hausdorff dimension of the resulting tree is then $(l+1)/l$.
Our extended setting now also includes the \emph{Brownian looptree}, defined in \cite{curien_scaling_2015}, which appears as the scaling limit of the so-called \emph{discrete looptree} associated with Barab\'asi-Albert model. This random metric space also has a natural construction through an aggregation of circles, and our theorem proves that this object has almost surely Hausdorff dimension $2$.
These examples do not really use our theorem in its full generality since their underlying block is deterministic. In fact, Hypothesis~\ref{gluing:hypothese d} is very general and is satisfied (for the appropriate $d\geq0$) by many distributions of blocks, including the Brownian CRT ($d=2$), see \cite{duquesne_exceptionally_2014}, the Brownian map ($d=4$), see \cite{serlet_large_1997,legall_geodesics_2010}, the $\theta$-stable trees ($d=\frac{\theta+1}{\theta}$), see \cite{duquesne_decomposition_2017}. Hence, our results can apply to a whole variety of such constructions, with a very general distributions for the blocks, and we are currently working on some examples in which this construction naturally arises as the limit of discrete models.

\paragraph{Indications on the proofs}
The computations of the dimension in Theorem~\ref{gluing:theoreme principal} differ, depending on the assumptions we make on $\alpha$ and $\beta$, and always consist of an upper bound, that we derive by providing explicit coverings, and a lower bound that arises from the construction of a probability measure satisfying the assumptions of Frostman's lemma, see Lemma~\ref{gluing:frostman lemma} in the Appendix for a statement. 

If we just assume that the scaling factors are smaller than $n^{-\alpha+\petito{1}}$, we can prove that the dimension is bounded above by $\frac{1}{\alpha}$ for rather general behaviours of the weights. To do so, we adapt arguments from \cite{curien_random_2017} to our new setting. The essential idea behind the proof is that the sub-structure descending from a block $\bb_n$ has size $n^{-\alpha+\petito{1}}$, and so that one only needs to cover every block $\bb_n$ with a ball of radius $n^{-\alpha+\petito{1}}$ to cover the whole structure.

When $\alpha<\frac{1}{d}$ and $\beta>1$, although the sub-structure descending from a block $\bb_n$ may have diameter of order $n^{-\alpha+\petito{1}}$, we can also check that the index of the first block glued on block $n$ has index roughly $n^{\beta}$, which is large compared to $n$. Hence the diameter of the substructure descending from $\bb_n$ is essentially due to $\bb_n$ itself. This gives us a hint that we can cover the whole substructure descending from the block $\bb_n$, using a covering of $\bb_n$ with balls that are really small compared to the size of $\bb_n$, and that it would lead to a more optimal covering. In fact we use these two observations to recursively construct a sequence of finer and finer coverings, which lead to the optimal upper-bound. The idea of the proof is presented in more details in Section~\ref{gluing:subsubsec:idea of the proof, upperbound}.

Concerning the lower bounds, for all values of $\alpha$ and $\beta$, we can define a natural probability measure $\bar{\mu}$ on $\cT$ as the limit of (a normalised version of) the measure $\mu_n$ defined on $\cT_n$ for every $n\geq 1$, see Section~\ref{gluing:height of a uniform point}. In the case $\beta\leq1$, this probability measure only charges the leaves of $\cT$, and an application of Lemma~\ref{gluing:frostman lemma} gives the lower bound $\frac{1}{\alpha}$.

For $\beta>1$, the measure $\bar{\mu}$ does not charge the leaves and so the preceding argument does not work. We construct another measure as the sub-sequential limit of a sequence of measures $(\pi_k)$ which are concentrated on sets of the form $\left(\cT_{2n_k}\setminus\cT_{n_k}\right)$ with $(n_k)$ chosen appropriately, see Section~\ref{gluing:subsubsec:idea of the proof} for a presentation of the idea of the proof. The limiting measure is then concentrated on a strict subset of leaves and again, using Lemma~\ref{gluing:frostman lemma} yields the appropriate lower bound.

\paragraph{Related constructions}
Let us also cite some other models that have been studied in the literature and which share some features with ours. First, the line-breaking construction of the scaling limit of critical random graphs in \cite{addario-berry_critical_2010} by Addario-Berry, Broutin and Goldschmidt, that of the stable trees in \cite{goldschmidt_line_2015} by Goldschmidt and Haas, and that of the stable graphs in \cite{goldschmidt_stable_2018} by Goldschmidt, Haas and Sénizergues, use a gluing procedure that is identical to ours. Their constructions are not directly handled by Theorem~\ref{gluing:theoreme principal} but they fall in a slightly more general setting, for which our proofs still hold. 
In \cite{borovkov_asymptotic_2006}, Borovkov and Vatutin study a discrete tree constructed recursively, which corresponds to the "genealogical tree" of the blocks in our model. 
Last, in \cite{rembart_recursive_2018}, Rembart and Winkel study the distribution of random trees that satisfy a self-similarity condition (in law). They provide an iterative construction of those trees in which infinitely many branches are glued at each step. 

\paragraph{Plan of the paper}
In Section~\ref{gluing:sec:general framework}, we give a rigorous definition of our model, set up some useful notation, and discuss some general properties. In the second section, we study the (normalised) natural measure $\bar{\mu}_n$ on $\cT_n$ and prove that it converges to a measure $\bar{\mu}$ on $\cT$ under suitable assumptions. In Section~\ref{gluing:subsec:upper bounds}, we prove the almost sure compactness of $\cT$ and some upper-bounds on its Hausdorff dimension under some relatively weak hypotheses. In Section~\ref{gluing:subsec:upper bound difficile}, we develop a new (more involved) approach that allows us to obtain a better upper-bound for some parameters for which the former fails to be optimal. In Section~\ref{gluing:sec: lower bounds}, we prove the lower bounds that match the upper-bounds obtained in Section~\ref{gluing:sec:upper bounds}. It is again divided in two subsections, each providing a proof that is only valid for some choices of parameters $\alpha$ and $\beta$. The Appendix \ref{gluing:subsec:hausdorff dimension} contains a short reminder of basic properties concerning Hausdorff dimension. The Appendices \ref{gluing:app:embedded construction}, \ref{gluing:subsec:decomposition in fragments} and \ref{gluing:subsec:computations} contain some technical proofs that can be skipped at first reading.
\paragraph{Acknowledgements}The authors would like to thank the anonymous referees for their valuable comments which helped to improve the presentation of the manuscript.
\section{General framework}\label{gluing:sec:general framework}
In this section, we start by providing a precise definition of our model and then we investigate some of its general properties.
\subsection{Construction}\label{gluing:subsec:construction}

Consider $((\bb_n,\bd_n,\brho_n,\bnu_n))_{n\geq1}$ a sequence of compact pointed metric spaces endowed with a finite Borel measure. Recall from the introduction the heuristics of our recursive construction. We define $\cT_1$ as the first block $\bb_1$ endowed with its measure $\bnu_1$. Then at each step, we construct $\cT_{n+1}$ from $\cT_n$ by gluing the root of the block $\bb_{n+1}$ to a random point $X_n\in\cT_n$, which has distribution (a normalised version of) $\mu_n$. The measure $\mu_{n+1}$ is defined as the sum of the measures $\mu_n$ and $\bnu_{n+1}$, the measure supported by $\bb_{n+1}$. We define $\cT^*$ as the increasing union of all the $\cT_n$ for $n\geq 1$, and its completion is denoted $\cT$.
In the next paragraph, we describe formally how to construct such growing metric spaces as subsets of a larger ambient space.
The definitions here are rather technical and the proofs in the paper do not use the details of the construction, so the reader can skip this part at first reading.
\paragraph{Embedded construction} We consider $(U,\delta)$ the Urysohn space, and fix a point $u_0\in U$. The space $U$ is defined as the only Polish metric space (up to isometry) which has the following extension property (see \cite{husek_urysohn_2008} for constructions and basic properties of $U$): given any finite metric space $X$, and any point $x\in X$, any isometry from $X\setminus\{x\}$ to $U$ can be extended to an isometry from $X$ to $U$.
In the rest of the construction, we assume that the measured metric spaces $((\bb_n,\bd_n,\brho_n,\bnu_n))_{n\geq1}$ are all embedded in the space $U$ and that their root is identified to $u_0$. From the properties of the Urysohn space, this is always possible (see Appendix \ref{gluing:app:embedded construction} for a construction in the case of random blocks).

We introduce
\[\ell^1(U,u_0):=\enstq{(x_n)_{n\geq1}\in U^{\N^*}}{\sum_{n=1}^\infty \delta(x_n,u_0)<+\infty }.\]
If we endow $\ell^1(U,u_0)$ with the distance $\dist((x_n)_{n\geq1},(y_n)_{n\geq1})=\sum_{n=1}^\infty \delta(x_n,y_n)$, it is an easy exercise to see that it makes this space Polish. We can now construct the $\cT_n$ recursively, by $\cT_1=\enstq{(x,u_0,u_0,\dots)}{x\in\bb_1}$, and identifying $\cT_1$ to the block $\bb_1$, we set $\mu_1=\bnu_1$. 
For $n\geq 1$, the point $X_n$ is sampled according to $\bar{\mu}_n$ a normalised version of $\mu_n$  with total mass $1$. The point $X_n$ is of the form $\left(x_1^{(n)},x_2^{(n)}, \dots, x_n^{(n)}, u_0, \dots\right)$ and we set
\[\cT_{n+1}:=\cT_n\cup \enstq{\left(x_1^{(n)},x_2^{(n)}, \dots x_n^{(n)},x, u_0 \dots\right)}{x\in \bb_{n+1}}.\]
We set $\mu_{n+1}:=\mu_n+\bnu_n$, where as in the preceding section, we see $\bb_{n+1}$ as the corresponding subset of $\cT_{n+1}$. 
Then $\cT^*=\bigcup_{n\geq 1} \cT_n$ and $\cT=\overline{(\cT^*)}$ is its closure in the space $(\ell^1(U,u_0),\dist)$. At the end $\cT$ is a random closed subset of a Polish space.

In the rest of the paper, we will not refer to this formal construction of $\cT$ and we will identify $\bb_n$ with the corresponding subset in $\cT$. We recall the notation $W_n=\sum_{k=1}^n w_n$ for the total mass of the measure $\mu_n$.

\subsection{Some notation}\label{gluing:notation}
\begin{figure}
	\begin{center}
		\includegraphics[height=6cm]{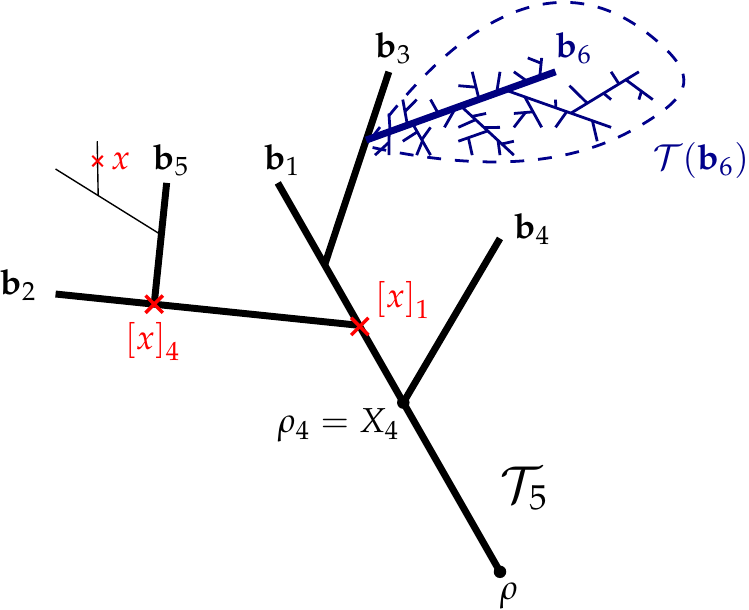}
		\caption{Substructure descending from a set, and projection on a substructure, illustrating some notation introduced in the paper in the case of the gluing of segments.}\label{gluing:fig:notation}
	\end{center}
\end{figure}
Let us introduce some notation that will be useful in the sequel, some of which is illustrated in Figure \ref{gluing:fig:notation}. Recall that from now on, we always assume that the blocks are of the form \eqref{gluing:block scaling}.

$\bullet$ If $(E,\dist,\rho)$ is a pointed metric space, and $x\in E$, we define $\haut(x)$, the height of $x$, as its distance to the root $\dist(\rho,x)$. We also denote $\haut(E)=\sup_{x\in E}\haut(x)$, the height of $E$. Let us consider $(\bB,\bD,\bRho,\bNu)$, a random block of our model before scaling, and $X$ a point of $\bB$ which conditionally on $(\bB,\bD,\bRho,\bNu)$, has distribution $\bNu$. We denote
\begin{equation}\label{gluing:hauteur point unif}\cH:=\haut(X)=\bD(\bRho,X),
\end{equation}
the height of a uniform random point in the block. Remark that Hypothesis~\ref{gluing:hypothese d} implies that $\Ec{\cH^2}<\infty$, and that $\Pp{\cH>0}>0$. Some of our results are stated under these weaker assumptions. 

$\bullet$ Whenever we sample the point $X_n$ under $\bar{\mu}_n$, we do it in the following way: first we sample $K_n$ such that for all $1\leq k \leq n$, $\Pp{K_n=k}=\frac{w_k}{W_n}$ and then, conditionally on $K_n=k$, the point $X_n$ is chosen on the block $\bb_k$ using the normalised version of the measure $\bnu_k$. Whenever $K_n=k$, we say that $\bb_{n+1}$ is grafted onto $\bb_k$ and write $\bb_{n+1}\rightarrow\bb_k$. Remark that this entails that $X_n\in \bb_k$, but this condition is not sufficient in the case where $X_n$ belongs to several blocks (which only happens if the measures carried by the blocks have atoms). We denote
\begin{align}\label{gluing:eq:def of mu_n*}
\bar{\mu}_n^*:=\text{law of } (K_n,X_n), 
\end{align}
seen as a measure on $\bigsqcup_{k=1}^n\{k\}\times \bb_k$. In this way, the random variables $((K_n,X_n))_{n\geq 1}$ are independent with respective distributions $(\bar{\mu}_n^*)_{n\geq 1}$.
We remain loose on the fact that we sometimes consider the blocks as abstract metric spaces and at other times we see them as subsets of $\cT$. 
It is implicit in the preceding discussion that everything is expressed conditionally on the sequence of blocks $(\bb_n)_{n\geq 1}$.

$\bullet$ We simultaneously construct a sequence of increasing discrete trees $(\mathsf{T}_n)_{n\geq 1}$ by saying that for $n\geq 1$, the tree $\mathsf{T_n}$ has $n$ nodes labelled $1$ to $n$ and $i$ is a child of $j$ if and only if $\bb_i\rightarrow\bb_j$. Also define $\mathsf{T}$ their increasing union. We denote $\prec$ the genealogical order on $\N^*$ induced by this tree. We denote $\dist_\mathsf{T}(i,j)$ for the graph distance between the nodes with label $i$ and $j$ in this tree and $\haut_\mathsf{T}(\cdot)$ for their height.

$\bullet$ For $x\in\cT$, we define $[x]_n$, the projection of $x$ on $\cT_n$, as the unique point $y$ of $\cT_n$ that minimizes the distance $\dist(x,y)$.  

$\bullet$ Similarly, for $k\geq 1$, we define $[k]_n$, the projection of $k$ on $\mathsf{T}_n$, as the unique node $i\leq n$ that minimizes the distance $\dist_\mathsf{T}(i,k)$. 

$\bullet$ If $S$ is a subset of a block $\bb_n$ for some $n\geq 1$ then we define $\cT(S)$, the substructure descending from $S$ as
\[\cT(S):=\overline{S\cup \underset{[X_{i-1}]_n\in S}{\bigcup_{i\succ n}}\bb_i}.\]
If $S=\bb_n$, this reduces to
\[\cT(\bb_n)=\overline{\bigcup_{i\succeq n}\bb_i},\]
and we consider $(\cT(\bb_n),\dist,\brho_n)$ as a rooted metric space.

$\bullet$ Remark that if $x\in\cT(\bb_k)$ for some $k\geq 1$ then we have $[x]_k\in\bb_k$ and more generally, for any $n\leq k$, we have $[x]_n\in \bb_{[k]_n}$. 

$\bullet$ We often use the little-o notation and denote $\petito{1}$ a deterministic function that tends to $0$ when some parameter tends to $0$ or $\infty$, depending on the context. For such functions that are random, we write instead $\petitoom{1}$.
\subsection{Zero-One law for compactness, boundedness and Hausdorff dimension}\label{gluing:properties of interest}

The main properties of $\cT$ and $\cL$ that we study are compactness and Hausdorff dimension. One can check that some of these properties are constants almost surely by an argument using Kolmogorov's zero-one law. 

Indeed, take the whole construction $\cT$ and contract the compact subspace $\cT_n$ into a single point. We can easily check that the resulting space is compact (resp.\ bounded) iff the former is compact (resp.\ bounded). Also, the subset $\cL$ and its image after the contraction of $\cT_n$ have the same Hausdorff dimension.
Now remark that the space that we just described only depends on the randomness of the blocks and the gluings after $n$ steps. Indeed, if we start at time $n$ with a unique point with weight $W_n$ and then follow the procedure by gluing recursively $\bb_{n+1}, \ \bb_{n+2}, \dots $, we get exactly the same space.

Hence, as this is true for all $n$, these properties only depend on the tail $\sigma$-algebra generated by the blocks and the gluings, and are therefore satisfied with probability $0$ or $1$.

\begin{remark}
	In the setting of \cite{curien_random_2017}, where the blocks are segments and the weights correspond to the lengths of those segments, the authors proved that the event of boundedness and compactness for $\cT$ coincide almost surely. This is not the case in our more general setting: consider the case of branches with weights and lengths defined as
	\begin{align*}w_n&=2^n, \ \lambda_n=2^{-n} \quad \text{for $n\notin\enstq{2^k}{k\in\N}$},\\
	w_{2^k}&=1, \ \lambda_{2^k}=1, \quad \text{for $k\in\N$.}
	\end{align*}
	In this case, an application of Borel-Cantelli lemma shows that a.s.\ for $n$ large enough, no branch $\bb_n$ is ever grafted onto a branch $\bb_{2^k}$ for any $k$. It is then clear that the resulting tree is a.s.\ bounded since the sum of the lengths of the branches $\bb_n$ for $n\notin\{2^k \mid k\in\N \}$ is finite, but it cannot be compact since there exists an infinite number of branches with length $1$.
\end{remark}
\subsection{Monotonicity of Hausdorff dimension}\label{gluing:monotonicity of Hausdorff dimension}
Let us present an argument of monotonicity of the Hausdorff dimension of $\cL$ with respect to the sequence $(\lambda_n)$, on the event on which $\cT$ is compact. 
Let $(w_n)$ be a sequence of weights and $(\lambda_n)$ and $(\lambda'_n)$ be two sequences of scaling factors such that for all $n\geq 1$, we have $\lambda_n\geq \lambda'_n$. 
Suppose that $\left((\bB_n,\bD_n,\bRho_n,\bNu_n)\right)_{n\geq1}$ is a sequence of random compact metric spaces endowed with a probability measure. Then, let $\cT$ (resp.\ $\cT'$) be the structure constructed using the blocks $(\bb_n,\bd_n,\brho_n,\bnu_n)=(\bB_n,\lambda_n\cdot\bD_n,\bRho_n,w_n\cdot \bNu_n)$, for $n\geq 1$, (resp.\ $(\bb_n',\bd_n',\brho_n',\bnu_n')=(\bB_n,\lambda'_n\cdot\bD_n,\bRho_n,w_n\cdot \bNu_n))$. Note that since we use the same sequence of weights we can couple the two corresponding gluing procedures. 

Let $f$ be the application that maps each of the block $\bb_n$ to the corresponding $\bb'_n$. Recall here that we see the blocks as subsets of the structure. We can verify that $f:\cT^*\longrightarrow(\cT')^*$, is $1$-Lipschitz. We can then extend uniquely $f$ to a function $\hat{f}:\cT\longrightarrow\cT'$, which is also $1$-Lipschitz. Suppose $\cT$ is compact. Then its image $\hat{f}(\cT)$ is compact, hence closed in $\cT'$. Since $(\cT')^*\subset \hat{f}(\cT)$ and $(\cT')^*$ is dense in $\cT'$, we have $\hat{f}(\cT)=\cT'$ and so $\hat{f}$ is surjective.
Now since $(\cT')^*=\hat{f}(\cT^*)$, we also have $\cL'=\hat{f}(\cL)$, and since $\hat{f}$ is Lipschitz,
\begin{equation}\label{gluing:monotonicity}
\mathrm{dim_H}(\cL')\leq \mathrm{dim_H}(\cL).\end{equation}

\section{Study of a typical point}\label{gluing:height of a uniform point}
In this section we study the height of a typical point of $\cT_n$, i.e.\ the distance from the root to a point sampled according to $\bar{\mu}_n$. The proofs in this section are really close to those of \cite[Section~1]{curien_random_2017}, to which we refer for details.
\subsection{Coupling with a marked point}\label{gluing:sec:convergence of the marked point}
We construct a sequence of points $(Y_n)_{n\geq 1}$ coupled with the sequence $(\cT_n)_{n\geq 1}$ in such a way that for all $n\geq 1$, the point $Y_n$ has distribution $\bar{\mu}_n$ conditionally on $\cT_n$ and such that the distance from $Y_n$ to the root is non-decreasing in $n$. 
For technical reasons, we in fact define a sequence $((J_n,Y_n))_{n\geq 1}$ such that for any $n\geq 1$, $(J_n,Y_n)$ has distribution $\bar{\mu}_n^*$ conditionally on $(\mathsf{T}_n,\cT_n)$, see \eqref{gluing:eq:def of mu_n*}.
The properties of this construction are stated in the following lemma.
\begin{lemma}\label{gluing:lem:properties of the marked point}
	We can couple the construction of $((\mathsf{T}_n,\cT_n))_{n\geq 1}$ with a sequence $((J_n,Y_n))_{n\geq 1}$ such that for all $n\geq 1$,
	\begin{enumerate}
		\item we have $J_n\in\{1,\dots,n\}$ and $Y_n\in \bb_{J_n}$,
		\item conditionally on $(\mathsf{T}_n,\cT_n)$, the couple $(J_n,Y_n)$ has distribution $\bar{\mu}_n^*$,
		\item for all $1\leq k\leq n$, $([J_n]_k,[Y_n]_k)=(J_k,Y_k)$.
	\end{enumerate}
	Furthermore, under the assumption \ref{gluing:hypothese d}\ref{gluing:moment exponentiel 1}, the sequence $(Y_n)_{n\geq 1}$ almost surely converges in $\cT$ iff 
	\begin{equation}\label{gluing:convergence des deux series}\sum_{n=1}^\infty \frac{\lambda_n w_n}{W_n}\ind{\lambda_n\leq 1}< \infty \quad \text{and} \quad \sum_{n=1}^\infty \frac{w_n}{W_n}\ind{\lambda_n> 1}<\infty.
	\end{equation}
\end{lemma}
Note that if either 
\begin{equation*}
W_\infty:=\sum_{n=1}^\infty w_n<\infty \quad \text{or} \quad \sum_{n=1}^\infty \frac{w_n \lambda_n}{W_n}<\infty,
\end{equation*} then \eqref{gluing:convergence des deux series} is satisfied, and this is the case under the assumptions of Theorem~\ref{gluing:theoreme principal}. In this case we let
\[Y:=\lim_{n\rightarrow\infty}Y_n.\]
\begin{proof}
	Let $n\geq 2$. Conditionally on $\cT_n$ and $\mathsf{T}_n$, sample a couple $(J_n,Y_n)$ under the measure $\bar{\mu}_n^*$.
	Then two cases may happen:
	\begin{itemize}
		\item with probability $1-w_n /W_n$ : we have $J_n<n$, so the point $Y_n$ belongs to $\cT_{n-1}$, that is $[Y_n]_{n-1} = Y_n$ , and
		conditionally on this event $([J_n]_{n-1},[Y_n]_{n-1})$ has the same distribution as $(J_{n-1},Y_{n-1})$,
		\item with probability $w_n /W_n$ : we have $J_n=n$. In this case the point $Y_n$ is located on the last block $\bb_n$ grafted on $\cT_{n-1}$. Conditionally on this event (if $w_n>0$), $Y_n$ is distributed on this block under the measure $\bnu_n$ and the couple $([J_n]_{n-1},[Y_n]_{n-1})$ is independent of the location of $Y_n$ on the $n$-th block and has the same distribution as $(J_{n-1},Y_{n-1})$.
	\end{itemize}
	From this observation we deduce that 
	\[(\cT_{n-1},\mathsf{T}_{n-1},[J_n]_{n-1}, [Y_n]_{n-1}) = (\cT_{n-1},\mathsf{T}_{n-1},J_{n-1}, Y_{n-1})\] in distribution and more generally,
	$(\cT_{k},\mathsf{T}_{k},[J_n]_{k}, [Y_n]_{k}) = (\cT_{k},\mathsf{T}_{k},J_k, Y_{k})$ in distribution for all $1 \leq k \leq n$.
	
	Reversing this observation, we can construct a sequence $(J_n,Y_n)_{n\geq1}$ (coupled to the $K_n$ and $X_n$ involved in the construction of $\cT^*$) such that conditionally on $\cT_n$ and $\mathsf{T}_n$, the couple $(J_n,Y_n)$ has distribution $\bar{\mu}^*_n$ and that for all $1\leq k \leq n$, we have $([J_n]_k,[Y_n]_k)=(J_k,Y_k)$.
	To do so, we consider:
	\begin{itemize}
		\item  a sequence $(U_n)_{n\geq 1}$ of uniform random variables on $\intervalleoo{0}{1}$,
		\item a sequence $(Z_n)_{n\geq 1}$ of points respectively sampled on $(\bb_n)_{n\geq 1}$ with respective distribution (a normalised version of) the measure $(\bnu_n)_{n\geq 1}$ whenever it is non-zero, (set $Z_n=\brho_n$ a.s.\ whenever $\bnu_n$ is trivial),
		\item a sequence $(I_n,P_n)_{n\geq 1}$, sampled with respective distributions $(\bar{\mu}_n^*)_{n\geq1}$,
	\end{itemize}  independently for all these random variables. Then we construct $(K_n,X_n)$ and $(J_n,Y_n)$ as follows. We set $(J_1,Y_1)=(1,Z_1)$. Then recursively for $n\geq 1$, we assume that $X_{n-1}$ (if $n\neq 1$) and $Y_n$ have been constructed:
	\begin{itemize}
		\item if $U_{n+1}\leq \frac{w_{n+1}}{W_{n+1}}$, then we set $(K_n,X_n):=(J_n,Y_n)$, $J_{n+1}:=n+1$ and $Y_{n+1}:=Z_{n+1}$,
		\item if $U_{n+1}>\frac{w_{n+1}}{W_{n+1}}$, then we set $(K_n,X_n):=(I_n,P_n)$, $J_{n+1}:=J_n$ and $Y_{n+1}:=Y_n$.
	\end{itemize}
	
	We can check that with this construction, for all $1\leq k \leq n$, we have $([J_n]_k,[Y_n]_k)=(J_k,Y_k)$, the $(K_n,X_n)_{n\geq 1}$ are independent with the appropriate distribution and for all $n\geq1$ conditionally on $\cT_n$ and $\mathsf{T}_n$ the couple $(J_n,Y_n)$ has distribution $\bar{\mu}^*_n$.
	Notice that the distance from $Y_n$ to the root $\rho$ is non-decreasing. Denoting $\cT_0=\{\rho\}$ and $Y_0=\rho$, for all $0\leq m\leq n$ we have
	\begin{align}\label{gluing:eq:croissance du point uniforme}
	\dist(Y_n,Y_m)=\dist(Y_n,\cT_m)=\sum_{k=m+1}^n \dist(Z_k,\rho_k)\ind{U_{k}\leq \frac{w_{k}}{W_{k}}},
	\end{align}
	which is equal in distribution to \[\sum_{k=m+1}^n \lambda_k \cH_k \ind{U_{k}\leq \frac{w_{k}}{W_{k}}},\]
	where the $(\cH_k)_{k\geq 1}$ are i.i.d.,  independent of the $(U_k)_{k\geq1}$ and have the law of $\cH$, see \eqref{gluing:hauteur point unif}.
	Under \ref{gluing:hypothese d}\ref{gluing:moment exponentiel 1} the random variable $\cH$ has a finite second moment, and an application of Kolmogorov's three series theorem tells us that the almost sure convergence of $\sum_{k\geq 1} \lambda_k \cH_k \ind{U_{k}\leq \frac{w_{k}}{W_{k}}}$ is equivalent to \eqref{gluing:convergence des deux series}. In this case,  $(Y_n)_{n\geq 1}$ is a Cauchy sequence in the complete space $\cT$ and hence it converges.
\end{proof}
Also notice that with this construction, the discrete counterpart of \eqref{gluing:eq:croissance du point uniforme} is
\begin{align}\label{gluing:eq:croissance hauteur point marque}
\dist_\mathsf{T}(J_n,J_m)=\sum_{k=m+1}^{n}\ind{U_k\leq \frac{w_k}{W_k}} \quad \text{and so} \quad \haut_\mathsf{T}(J_n)=\sum_{k=2}^{n}\ind{U_k\leq \frac{w_k}{W_k}}.
\end{align}
Remark that for any $\theta\in\R$,
\begin{align}\label{gluing:eq:hauteur discrete moment exponentiels}
\Ec{\exp\left(\theta\haut_\mathsf{T}(J_n)\right)}&=\Ec{\exp\left(\theta\sum_{k=2}^{n}\ind{U_k\leq \frac{w_k}{W_k}}\right)}\notag \\
&=\prod_{k=2}^n \left(\frac{W_k-w_k}{W_k}\cdot 1+\frac{w_k}{W_k}e^{\theta}\right)\notag \\
&=\exp\left(\sum_{k=2}^{n}\log\left(1+(e^\theta-1)\frac{w_k}{W_k}\right)\right)\notag \\
&\leq \exp\left((e^\theta-1)\sum_{k=2}^{n}\frac{w_k}{W_k}\right),
\end{align}
where in the last line we use the inequality $\log(1+x)\leq x$, valid for all $x>-1$. 
\subsection{Convergence \texorpdfstring{of the measure $\bar{\mu}_n$}{of the normalised weight measure}}
\begin{proposition}\label{gluing:convergence of the measure mu}
	Assume that $\Ec{\cH^2}<\infty$ and $\Pp{\cH>0}>0$ and that \eqref{gluing:convergence des deux series} holds. Then almost surely there exists a probability measure $\bar{\mu}$ on $\cT$ such that
	\[\bar{\mu}_n\underset{n\rightarrow\infty}{\longrightarrow} \bar{\mu} \qquad \text{weakly.}\]
	Furthermore, conditionally on $(\cT,\bar{\mu})$, the point $Y$ is distributed according to $\bar{\mu}$ almost surely. If $W_\infty<\infty$, then $\bar{\mu}=\frac{1}{W_\infty}\mu_\infty$, and $\bar{\mu}$ is concentrated on $\cT^*$. If $W_\infty=\infty$, then $\bar{\mu}$ is concentrated on $\cL$.
\end{proposition}
The proof of the last proposition is very similar to the proof of \cite[Theorem~4]{curien_random_2017}, and is left to the reader. We can easily check that the assumptions of Proposition~\ref{gluing:convergence of the measure mu} are satisfied under the hypotheses of Theorem~\ref{gluing:theoreme principal}. We now state an additional lemma that will be useful later in the paper.
\begin{lemma}\label{gluing:urne de polya}
	Suppose that the assumptions of Proposition~\ref{gluing:convergence of the measure mu} hold, that $\bar{\mu}$ is concentrated on the set $\cL$ and that the sequence of weights satisfies $\frac{w_n}{W_n}\leq n^{-1+\petito{1}}$. Then almost surely, we have
	\[\bar{\mu}(\cT(\bb_n))\leq n^{-1+\petitoom{1}},\]
	where the random function $\petitoom{1}$ is considered as $n \rightarrow\infty$.
\end{lemma}
\begin{proof}
	Let us introduce some notation. If $i\geq n$, we set $M^{(n)}_i:=\bar{\mu}_i(\cT(\bb_n))$ the relative mass of the tree descending from $\bb_n$ in $\cT_i$. As $i$ varies, this sequence of random variables evolves like one of Pemantle's time-dependent P\'olya urns (see \cite{pemantle_time_1990}) and is therefore a martingale. 
	The topological boundary of $\cT(\bb_n)$ in $\cT$ is either the empty set or the singleton $\{\brho_n\}$, thus it has zero $\bar{\mu}$-measure\footnote{Indeed, under the assumptions of the lemma, $\bar{\mu}$ is carried on the leaves.}. It follows from Portmanteau theorem that the quantity of interest $\bar{\mu}(\cT(\bb_n))$ corresponds to $M^{(n)}_\infty$, the almost sure limit of this positive martingale. 
	We can write 
	\[M^{(n)}_{i+1}=\left(\frac{W_i}{W_{i+1}}\right)M^{(n)}_i+\frac{w_{i+1}}{W_{i+1}}\ind{U_{i+1}\leq M^{(n)}_i},\]
	with $(U_{i})_{i\geq 1}$ a sequence of i.i.d.\ random variables, uniform on $\intervalleoo{0}{1}$. We are going to show by induction on $k\geq1$ that there exists a function $\petito{1}$ as $n\rightarrow\infty$ such that for all $i\geq n$, we have
	\[\Ec{(M^{(n)}_{i})^k}\leq n^{-k+\petito{1}}.\]
	Note that we use the notation $\petito{1}$ for all such functions, but that in this proof, the corresponding functions can depend on $k$ but not on $i$.
	
	$\bullet$ For $k=1$, the result follows from the fact that $(M^{(n)}_i)_{i\geq n}$ is a martingale and that almost surely $M^{(n)}_n\leq\frac{w_n}{W_n}\leq n^{-1+\petito{1}}$.
	
	$\bullet$ Let $k\geq 2$. Suppose that the result is true for all $1\leq l \leq k-1$. Then
	\begin{align*}
	& \ \Ecsq{\left(M^{(n)}_{i+1}\right)^k}{M^{(n)}_i}\\
	&=\Ecsq{ \left(\left(\frac{W_i}{W_{i+1}}\right)M^{(n)}_i+\frac{w_{i+1}}{W_{i+1}}\ind{U_{i+1}\leq M^{(n)}_i}\right)^k}{M^{(n)}_i}\\
	&= \left(\frac{W_i}{W_{i+1}}\right)^k \left(M^{(n)}_i\right)^k + \Ecsq{\sum_{l=0}^{k-1}\binom{k}{l}\left(\frac{W_i}{W_{i+1}}\right)^l \left(M^{(n)}_i\right)^{l}\left(\frac{w_{i+1}}{W_{i+1}}\ind{U_{i+1}\leq M^{(n)}_i}\right)^{k-l}}{M^{(n)}_i} \\
	&= \left(\frac{W_i}{W_{i+1}}\right)^k \left(M^{(n)}_i\right)^k + \sum_{l=0}^{k-1}\binom{k}{l}\left(\frac{W_i}{W_{i+1}}\right)^l \left(M^{(n)}_i\right)^{l+1}\left(\frac{w_{i+1}}{W_{i+1}}\right)^{k-l} \\
	&\leq \left(M^{(n)}_i\right)^k\left(\left(\frac{W_i}{W_{i+1}}\right)^k +k\cdot \frac{w_{i+1}}{W_{i+1}}\left(\frac{W_i}{W_{i+1}}\right)^{k-1}\right)  + \sum_{l=0}^{k-2}\binom{k}{l}\left(M^{(n)}_i\right)^{l+1}\left(\frac{w_{i+1}}{W_{i+1}}\right)^{k-l}.
	\end{align*}
	Now taking the expectation and using the fact that $\forall x\in\intervalleff{0}{1}, \ (1-x)^{k}+k(1-x)^{k-1}x\leq 1$, we get, using the induction hypothesis,
	\begin{align*}
	\Ec{(M^{(n)}_{i+1})^k}&\leq \Ec{(M^{(n)}_{i})^k}+\sum_{l=0}^{k-2}\binom{k}{l}\Ec{\left(M^{(n)}_i\right)^{l+1}}\left(\frac{w_{i+1}}{W_{i+1}}\right)^{k-l}\\
	&\leq \Ec{(M^{(n)}_{i})^k}+\sum_{l=0}^{k-2}\binom{k}{l}n^{-(l+1)+\petito{1}}(i^{-1+\petito{1}})^{k-l}.
	\end{align*}
	Summing over all $i$ we get that, for all $i\geq n$:
	\begin{align*}
	\Ec{(M^{(n)}_{i})^k}&\leq\Ec{(M^{(n)}_{n})^k}+\sum_{j=n}^\infty\sum_{l=0}^{k-2}\binom{k}{l}n^{-l-1+\petito{1}}j^{-k+l+\petito{1}}\\
	&\leq \Ec{(M^{(n)}_{n})^k}+\sum_{l=0}^{k-2}\binom{k}{l}(n^{-l-1+\petito{1}})\sum_{j=n}^\infty j^{-k+l+\petito{1}}\\
	&\leq n^{-k+\petito{1}}+\sum_{l=0}^{k-2}\binom{k}{l}n^{-l-1+\petito{1}}n^{-k+l+1+\petito{1}}\\
	&\leq n^{-k+\petito{1}}.
	\end{align*}
	This finishes the proof by induction.
	This property passes to the limit by dominated convergence so, for all $n\geq 1$, we have $\Ec{(M^{(n)}_{\infty})^k}\leq n^{-k+\petito{1}}$. For $N$ an integer and $\epsilon>0$,
	\begin{align*}
	\Pp{M^{(n)}_\infty\geq n^{-1+\epsilon}}&\leq n^{N-N\epsilon} \Ec{(M^{(n)}_{\infty})^N}\\
	&\leq n^{-N\epsilon+\petito{1}}.
	\end{align*}
	If we take $N$ large enough, those quantities are summable and so, using the Borel-Cantelli lemma we get that with probability one, $M^{(n)}_\infty\leq n^{-1+\epsilon}$ for all $n$ large enough. This completes the proof.
\end{proof}

\section{Upper-bounds and compactness for the $(\alpha,\beta)$-model}\label{gluing:sec:upper bounds}
In this section, we compute upper-bounds on the Hausdorff dimension of the set $\cL$. We first prove Proposition~\ref{gluing:prop:compacité+majoration dimension}, which tells us that, under the condition that $\lambda_n\leq n^{-\alpha+\petito{1}}$ for some $\alpha>0$ and in a very general setting for the behaviour of the weights $(w_n)$, the dimension is bounded above by $1/\alpha$. The techniques used in the proof are very robust, and do not depend on the geometry of the blocks nor on the sequence of weights.
In a second step, in Proposition~\ref{gluing:dimension}, we handle the more specific case where the underlying block satisfies Hypotheses \ref{gluing:hypothese d} and that $\lambda_n\leq n^{-\alpha+\petito{1}}$ for some $0<\alpha<1/d$ and $w_n\leq n^{-\beta+\petito{1}}$ for some $\beta>1$. In the proof of this proposition, a careful analysis allows us to refine some of the arguments of the previous proof and prove upper-bounds on the Hausdorff dimension of $\cL$ that are below the "generic" value $1/\alpha$, given by Proposition~\ref{gluing:prop:compacité+majoration dimension}. The techniques used for the proof are new and really take into account the behaviour of the weights and the geometry of the blocks.

\subsection{Upper-bound independent of the weights and compactness}\label{gluing:subsec:upper bounds}

Notice that under \ref{gluing:hypothese d}\ref{gluing:moment exponentiel 1}, the underlying block $(\bB,\bD,\bRho,\bNu)$ satisfies, for any $N>0$,
\begin{align*}
\Pp{\diam(B)\geq n^\epsilon}\leq \frac{\Ec{\diam(\bB)^N}}{n^{-N\epsilon}},
\end{align*}
which is summable if $N$ is large enough.
Hence if $(\bB_n)$ is an i.i.d.\ sequence with the same law as $\bB$, then using the Borel-Cantelli lemma we have almost surely,
\begin{equation}\label{gluing:diam o 1}
\diam(\bB_n)\leq n^{\petitoom{1}}.
\end{equation}
\begin{proposition}\label{gluing:majoration hausdorff} \label{gluing:prop:compacité+majoration dimension}
	Suppose $\lambda_n\leq n^{-\alpha+\petito{1}}$, with $\alpha>0$, and that for all $n$, we have $W_n\leq n^{\gamma}$ for some $\gamma>0$. Suppose also that \eqref{gluing:diam o 1} holds. Then the tree-like structure $\cT$ is almost surely compact and we have
	\begin{enumerate}
		\item $\mathrm{d_H}(\cT_{n},\cT)\leq n^{-\alpha+\petitoom{1}}$,\label{gluing:distance hausdorff tn}
		\item $\mathrm{dim_H}(\cL)\leq \frac{1}{\alpha}.$\label{gluing:majoration dim hausdorff}
	\end{enumerate}
\end{proposition}
Since our model is invariant by multiplying all the weights by the same constant, we can always assume that $w_1\leq 1$. Hence, the assumption in the lemma is always satisfied if $W_n$ grows at most polynomially in $n$, which is the case if Hyp.~\ref{gluing:condition beta grand}, Hyp.~\ref{gluing:condition beta 1} or Hyp.~\ref{gluing:condition beta petit} is fulfilled, for any choice of $\alpha>0$ and $\beta\in\R$.
\begin{proof}[Proof of Proposition~\ref{gluing:majoration hausdorff}]
	We start with point \ref{gluing:distance hausdorff tn}. First,
	\begin{align*}
	\mathrm{d_H}(\cT_{2^{i}},\cT_{2^{i+1}})&\leq \sup_{2^i+1\leq k \leq 2^{i+1}} \lambda_k \diam(\bB_k) + \sup_{2^i+1\leq k \leq 2^{i+1}}\dist(\rho_k,\cT_{2^i}).
	\end{align*}
	For any $2^i\leq k \leq 2^{i+1}-1$, the point $\rho_{k+1}$ in the tree is identified with the point $X_{k}$, taken under the measure $\bar\mu_{k}$ on the tree $\cT_{k}$. From our construction in Section~\ref{gluing:notation}, the point $X_{k}$ belongs to some $\bb_{K_{k}}$, and the couple $(K_k, X_k)$ is sampled with measure $\bar{\mu}_k^*$. 
	Bounding the contribution of every block along the ancestral line with their maximum, we get
	\begin{align}\label{gluing:eq:majoration distance Xk T2i}
	\dist(X_k,\cT_{2^{i}})\leq \left(\sup_{2^i+1\leq k \leq 2^i} \lambda_k \diam(\bB_k) \right) \haut_\mathsf{T}(K_k).
	\end{align}
	Now using Lemma~\ref{gluing:lem:croissance logarithmique} in Appendix~\ref{gluing:subsec:computations}, we know that there exists a constant $C>0$ such that $\sum_{i=1}^{n}\frac{w_i}{W_i}\leq C \log n$. Combining this with equation \eqref{gluing:eq:hauteur discrete moment exponentiels} (which holds for $K_k$ because it has the same distribution as $J_k$) and Markov inequality, we get for any $u>0$,
	\begin{align*}
	\Pp{\haut_\mathsf{T}(K_n)\geq u \log n}\leq \exp\left( (C(e-1)-u) \log n\right)=n^{C(e-1)-u}.
	\end{align*}
	The last display is summable in $n$ if we choose $u$ large enough. Hence using the Borel-Cantelli lemma, we almost surely have $\haut_\mathsf{T}(K_n)\leq u\log n$ for $n$ large enough. Hence, in \eqref{gluing:eq:majoration distance Xk T2i} we have $\haut_\mathsf{T}(K_k)=(2^i)^{\petitoom{1}}$. Combining this with \eqref{gluing:diam o 1} and the upper-bound on $\lambda_n$ we get,
	\begin{align*}
	\mathrm{d_H}(\cT_{2^{i}},\cT_{2^{i+1}})\leq (2^i)^{-\alpha +\petitoom{1}}.
	\end{align*}
	Replacing $i$ by $k$ and summing the last display over $k\geq i$,
	\[\sum_{k=i}^\infty \mathrm{d_H}(\cT_{2^{k}},\cT_{2^{k+1}})\leq (2^i)^{-\alpha + \petitoom{1}},\]
	hence the sequence of compact sets $(\cT_{2^{i}})$ is a.s.\ a Cauchy sequence for Hausdorff distance between compacts of the complete space $\cT$. So the sequence $(\cT_{n})_{n\geq1}$ is also Cauchy because of the increasing property of the construction, and $\cT$ is then almost surely compact. Moreover we have, a.s.\  \[\mathrm{d_H}(\cT_{2^{i}},\cT)\leq (2^i)^{-\alpha + \petitoom{1}},\] 
	and this entails \ref{gluing:distance hausdorff tn}.
	Remark that since $\haut(\cT(\bb_n)) \leq \mathrm{d_H}(\cT_{n-1},\cT)$, this implies that a.s.\ we have
	\begin{equation}\label{gluing:hauteur sous arbres}
	\haut\left(\cT(\bb_n)\right)\leq n^{-\alpha+\petitoom{1}}.
	\end{equation}
	
	We now prove point \ref{gluing:majoration dim hausdorff}. Let $\epsilon>0$. From \eqref{gluing:hauteur sous arbres}, the collection of balls $\Ball (\brho_n, n^{-\alpha+\epsilon})$, for $n\geq N$, where $N$ is an arbitrary number, is a covering of $\cL$ whose maximal diameter tends to $0$ as $N\rightarrow\infty$. Besides, if we fix $\delta$, for $N$ large enough, and $s>\frac{1}{\alpha - \epsilon}$, we have:
	\[\cH_s^\delta(\cL)\leq \sum_{n=N}^\infty \diam(\Ball (\brho_n, n^{-\alpha+\epsilon}))^s\leq \sum_{n=N}^\infty 2^sn^{(-\alpha+\epsilon)s}\underset{N\rightarrow\infty}{\longrightarrow} 0.\]
	Hence for all such $s$, we have $\cH_s(\cL)=0$ and so $\mathrm{dim_H}(\cL)\leq \frac{1}{\alpha - \epsilon}$. Letting $\epsilon\rightarrow 0$ finishes the proof.
\end{proof}

\subsection{Upper-bound for $\alpha<1/d$ and $\beta>1$}\label{gluing:subsec:upper bound difficile}
Now let us study the specific case where the blocks satisfy Hypothesis~\ref{gluing:hypothese d} and that $\lambda_n\leq n^{-\alpha+\petito{1}}$ for some $0<\alpha<1/d$ and $w_n\leq n^{-\beta+\petito{1}}$ for some $\beta>1$. The preceding Proposition~\ref{gluing:prop:compacité+majoration dimension} still holds but it is not optimal in this specific case. As in the previous proof we construct explicit coverings of the set $\cL$ in order to bound its Hausdorff dimension. We construct them using an iterative procedure, which strongly depends on the dimension $d$ and the exponent $\beta$. Starting from the covering given in the proof of Proposition~\ref{gluing:prop:compacité+majoration dimension}, the procedure provides at each step a covering that is "better" in some sense than the preceding. In the limit, we prove the bound given in Proposition~\ref{gluing:dimension}, which explicitly depends on $\beta$ and $d$.
\begin{proposition}\label{gluing:dimension}
	Suppose $0<\alpha<\frac{1}{d}$ and $\beta>1$ and that for all $n\geq1, \ \lambda_n\leq n^{-\alpha+\petito{1}}$ and $w_n\leq n^{-\beta+\petito{1}}$. Suppose also that \ref{gluing:hypothese d}\ref{gluing:moment exponentiel 1} and \ref{gluing:hypothese d}\ref{gluing:nombre mininimal de boules} hold for some $d\geq0$. Then the Hausdorff dimension of $\cL$ almost surely satisfies:
	\[\mathrm{dim_H}(\cL)\leq \frac{2\beta-1-2\sqrt{(\beta-1)(\beta-\alpha d)}}{\alpha}.\]
\end{proposition}
For our purposes, we will work with countable sets of balls of $\cT$, i.e.\ sets of the form
\[R=\enstq{\Ball\left(x_i,r_i\right)}{\forall i\geq1, \ x_i\in \cT, \ r_i>0},\]
where $\Ball(x,r)$ denotes the open ball centred at $x$ with radius $r$.
Let us introduce some notation.
If $R$ is such a set of balls of $\cT$, we say that $R$ is a \emph{covering} of the subset $X\subset \cT$ if $X\subset\bigcup_{B\in R} B$. We can also define the $s$-volume of $R$ as 
\[\Vol_s(R):=\sum_{B \in R} \diam(B)^s.\]
In this way if the diameters of the balls that belong to $R$ are bounded above by some $\delta>0$, and $R$ is a covering of $X$, then $\cH_s^\delta(X)\leq \Vol_s(R)$, see Section~\ref{gluing:subsec:hausdorff dimension} in the Appendix for the definition of  $\cH_s^\delta(X)$.
Also, if $R$ and $R'$ are collections of balls and $R$ covers $X$ and $R'$ covers $X'$, then obviously $R\cup R'$ is a countable set of balls that covers $X\cup X'$ and for any $s$, we have \begin{equation}
\Vol_s(R\cup R' )\leq \Vol_s(R)+\Vol_s(R').
\end{equation}
In what follows, we construct random sets of balls and we prove that they are coverings of our set $\cL$, which allow us to prove upper-bounds on the Hausdorff dimension of $\cL$.

\subsubsection{An idea of the proof}\label{gluing:subsubsec:idea of the proof, upperbound}
We briefly explain the idea of the proof before going into technicalities. The goal will be to provide a covering of each $\cT(\bb_n)$, for all $n$ large enough. Since from the definition of $\cL$ we have for any $N\geq1$,
\begin{equation}
\cL\subset\bigcup_{n\geq N}\cT(\bb_n), \label{gluing:leaves in subtrees}
\end{equation}
then the union over all $n$ large enough of coverings of the $\cT(\bb_n)$ is indeed a covering of $\cL$.

We recall how we derived the upper-bound $\frac{1}{\alpha}$ for the Hausdorff dimension of $\cL$ in the proof of Proposition~\ref{gluing:majoration hausdorff} \ref{gluing:majoration dim hausdorff}. The idea is to consider for every $n\geq 1$, a ball of radius $n^{-\alpha+\epsilon}$, say centred at $\brho_n$. For $n$ large enough, this ball covers $\cT(\bb_n)$ by \eqref{gluing:hauteur sous arbres}. Thanks to \eqref{gluing:leaves in subtrees}, the set of balls $\enstq{\Ball (\brho_n, n^{-\alpha+\epsilon})}{n\geq N}$, for any $N\geq 1$, is a covering of $\cL$.

For $\beta\leq 1$, this covering is good because, as a block of index $n$ has relative weight $w_n/W_n$ which can be of order up to $n^{-1+\petito{1}}$ when it appears, the indices of the first blocks that are glued on $\bb_n$ can have also an index of the order of $n$, and so a height of order up to $n^{-\alpha}$.
On the contrary, if $\beta>1$, we will see that the first block to be grafted on $\bb_n$ has index roughly of order $n^\beta$, and so a height at most of order $n^{-\alpha\beta}$, which is very small compared to $n^{-\alpha}$. This gives us a hint that we can provide a "better" covering using a big number of smaller balls to cover $\bb_n$ instead of just a "big" one, see Figure \ref{gluing:fig:step 2}.  We will use this rough idea to provide an algorithm which will construct finer and finer (random) coverings. Let us fix $\beta>1$ from now on and take $s>d$, and explain informally how the algorithm works.
\subparagraph*{Goal:}
At each step $i$ of the algorithm, we want to construct for all $n\geq 1$ a set of balls $R_{n,i}^s$ such that, for $n$ large enough, this set of balls is a covering of $\cT(\bb_n)$. Such a set of balls $R_{n,i}^s$ will have an $s$-volume of roughly $n^{f_i(s)}$, say. From step to step, we try to lower the $s$-volume of the set of balls constructed by the algorithm, which corresponds to lowering this exponent $f_i(s)$. Whenever we manage to get an exponent below $-1$, we stop the algorithm. We will see that it implies that the Hausdorff dimension of $\cL$ is lower or equal to $s$.
\subparagraph*{Step 1:}
The first step of the algorithm is deterministic and corresponds to what we did in the proof of Proposition~\ref{gluing:majoration hausdorff}. For each $n$ we take a ball centred at $\brho_n$ of radius roughly $n^{-\alpha}$ (in fact $n^{-\alpha+\epsilon}$ but let us not consider these technicalities for the moment). As seen before, for $n$ large enough, it is a covering of $\cT(\bb_n)$. The $s$-volume of this covering is then of order $n^{-\alpha s}$. Denote $f_1(s)=-\alpha s$. If $f_1(s)<-1$, stop. Otherwise, proceed to step $2$.
\subparagraph*{Step 2:}
As represented in Figure \ref{gluing:fig:tbn}, decompose $\cT(\bb_n)$ as
\[\cT(\bb_n)=\bb_n \cup \bigcup_{\bb_k\rightarrow\bb_n}\cT(\bb_k).\]
\begin{figure}
	\begin{center}
		\begin{tabular}{cccc}
			\subfloat[The substructure $\cT(\bb_n)$]{\includegraphics[height=3.5cm]{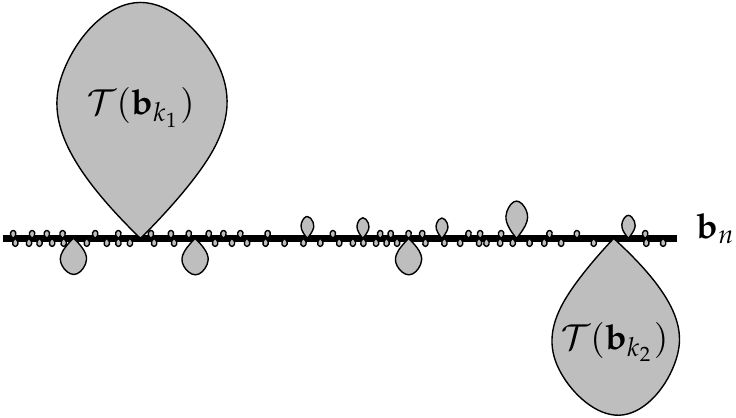}\label{gluing:fig:tbn}} & 			\subfloat[A covering of $\bb_n$ with small balls]{\includegraphics[height=3.5cm]{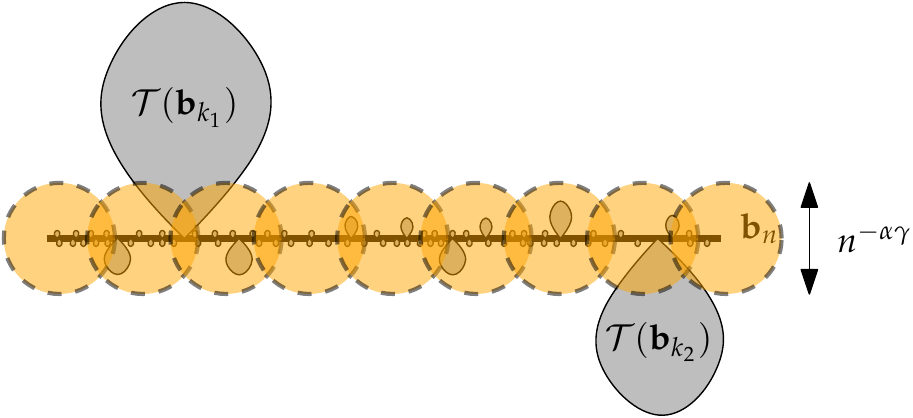}\label{gluing:fig:small balls}} \\
			\subfloat[The remaining substructures are covered using the preceding step]{\includegraphics[height=4.5cm]{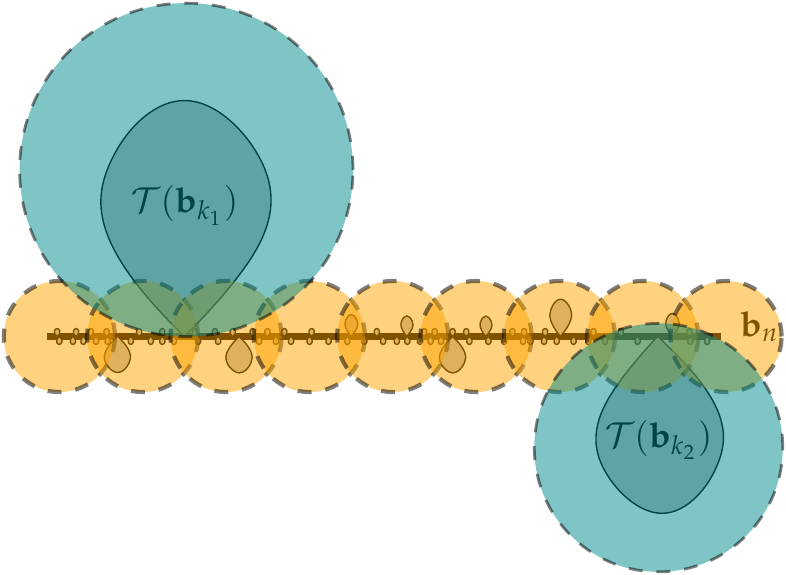}\label{gluing:fig:remaining subtrees}}
		\end{tabular}
		\caption{Explanation of Step 2 of the algorithm}\label{gluing:fig:step 2}
	\end{center}
\end{figure}
Since the first block grafted on the block $\bb_n$ has typically an index that is very large compared to $n$, we design a covering using smaller balls. We fix $\gamma>1$ and decide to cover $\bb_n$ with balls of size $n^{-\alpha\gamma}$, so that the blocks (and their descending substructure) of index $>n^{\gamma}$ are included in these balls, see Figure \ref{gluing:fig:small balls}. Since the blocks have dimension $d$, this covering uses roughly $\left(\frac{n^{-\alpha}}{n^{-\alpha\gamma}}\right)^d$ balls, each with $s$-volume $n^{-\alpha\gamma s}$. So the total volume used is around
$n^{-\alpha d + \alpha\gamma d - \alpha \gamma s}$.

But doing so, we forgot to cover the blocks $\bb_k$ such that $\bb_k\rightarrow \bb_n$ and $k\leq n^{\gamma}$. To take care of them, we use the preceding step of the algorithm and cover each of them with a ball of radius $k^{-\alpha}$, see Figure \ref{gluing:fig:remaining subtrees}. Recalling that $s\leq 1/\alpha$, we get that in expectation, these balls have a $s$-volume of order
\[\sum_{k=n+1}^{n^\gamma}\Pp{\bb_k\rightarrow\bb_n} k^{-\alpha s}\approx n^{-\beta}\sum_{k=n+1}^{n^\gamma} k^{-\alpha s}\approx n^{-\beta +\gamma(1-\alpha s)}.\]
Hence, the total $s$-volume used to cover $\cT(\bb_n)$ has order $n^{\max(-\beta +\gamma(1-\alpha s),-\alpha d + \alpha\gamma d - \alpha \gamma s)}$. Since we want to construct a covering having the smallest possible volume, we can optimize on $\gamma$ the last exponent. Under our assumptions, one can check that it is minimal if we take $\gamma:=\frac{\beta-\alpha d}{1-\alpha d}>1$. We then get 
\[\max(-\beta+\gamma(1-\alpha s),-\alpha d +\alpha\gamma d-\alpha\gamma s)=\frac{-\alpha d + \alpha \beta d - \alpha \beta s +\alpha^2 d s}{1-\alpha d}:=f_2(s).\] 
We can check that the new exponent $f_2(s)$ is smaller than $f_1(s)=-\alpha s$. Hence we can cover $\cT(\bb_n)$ with balls using a total $s$-volume of a lower order than the preceding step. If $f_2(s)<-1$, stop. Otherwise, proceed to step $3$.

\subparagraph*{Step $i$:}
Now we recursively repeat the preceding step. Thanks to step $i-1$, we know that we can provide a covering of $\cT(\bb_n)$ for any $n$, using a $s$-volume of approximately $n^{f_{i-1}(s)}$. Now we fix a number $\gamma>1$ and we cover the block $\bb_n$ with balls of radius $n^{-\alpha \gamma}$. 
As in step $2$, this covering has a $s$-volume of order $n^{-\alpha d + \alpha\gamma d - \alpha \gamma s}$. Then we take care of the $\bb_k$ such that $\bb_k\rightarrow \bb_n$ and $k<n^{\gamma}$. To cover them we use step $i-1$, which ensures that we can do that for each $k$ with a $s$-volume roughly $k^{f_{i-1}(s)}$. Hence the expectation on the $s$-volume for all these balls is, if $s$ is such that $f_{i-1}(s)\geq -1$,
\[\sum_{k=n+1}^{n^{\gamma}}\Pp{\bb_k\rightarrow\bb_n} k^{f_{i-1}(s)}\approx n^{-\beta}\sum_{k=n+1}^{n^{\gamma}} k^{f_{i-1}(s)}\approx n^{-\beta +\gamma(1+f_{i-1}(s))}.\]
We then choose the optimal $\gamma>1$ that minimizes the maximum of the exponents 
\[\max(-\alpha d + \alpha\gamma d - \alpha \gamma s,-\beta +\gamma(1+f_{i-1}(s))).\]
We denote $\gamma_i(s)$ the value for which the minimum is obtained, which depends on $s$.
The first exponent is linearly decreasing with $\gamma$, the other one is linearly increasing, and their value for $\gamma$ tending to $1$, satisfy $-\alpha s> -\beta+ 1+f_{i-1}(s)$. Hence, the value of $\gamma_i(s)$ is the value for which the two of them are equal, and this value is strictly greater than $1$. We call this minimal exponent $f_i(s)$. If $f_i(s)<-1$, stop. Otherwise, proceed to step $i+1$.

\subparagraph*{Upper-bound on Hausdorff dimension}
Now, suppose $s$ is such that $f_i(s)$ is well-defined and $f_i(s)<-1$, for some $i\geq1$. If we cover every $\cT(\bb_n)$ using the covering provided by step $i$ of the algorithm, then the union of all those coverings covers $\cL$. Furthermore, we only need to cover all the $\cT(\bb_n)$ for $n$ sufficiently large to cover $\cL$, so we can have a covering of $\cL$ using arbitrarily small balls. Hence we get that for all $\delta>0$, we have $\cH_s^\delta(\cL)<\sum_{n=1}^\infty n^{f_i(s)}<\infty$ and so $\cH_s(\cL)<\infty$, which proves that \[\mathrm{dim_H}(\cL)\leq s.\]
This rough analysis is turned into a rigorous proof in what follows. We begin with elementary definitions and calculations that arise from what precedes.

\subsubsection{Study of a sequence of functions}
We begin by defining recursively the sequence of functions $(f_i)_{i\geq1}$, together with a sequence $(s_i)_{i\geq1}$ of real numbers.
\begin{definition-proposition}\label{gluing:def:fonctions fi}
	We set $s_0:=\infty$. We define a sequence $(f_i)_{i\geq1}$ of functions as follows. We set 
	\[\forall s\in\intervallefo{d}{\infty},\quad f_1(s):=-\alpha s,\]
	and set $s_1:=\frac{1}{\alpha}$. Then for all $i\geq 1$, we recursively define:
	\[\forall s\in\intervalleff{d}{s_{i}}, \quad f_{i+1}(s):=\frac{\alpha(-d+\beta d-\beta s - f_i(s) d)}{1+f_i(s)+\alpha s-\alpha d}.\]
	Define $s_{i+1}$ as the unique solution to the equation $f_{i+1}(s)=-1$.
\end{definition-proposition}
Before proving the validity of this definition, let us state some properties of this sequence of functions.
\begin{proposition}\label{gluing:prop des fi}
	The following properties are satisfied:
	\begin{enumerate}
		\item  For all $i\geq1$, the function $f_i$ is continuous, strictly decreasing, and $f_i(d)=-\alpha d$.\label{gluing:continue}
		\item\label{gluing:decroissance}For all $i\geq 1$, for all $s\in\intervalleof{d}{s_{i}}$, we have $ f_{i+1}(s)<f_{i}(s)$.
		\item\label{gluing:convergence} Let $s_\infty:=\frac{2\beta-1-2\sqrt{(\beta-1)(\beta-\alpha d)}}{\alpha}$. Then we have for all $s\in\intervallefo{d}{s_\infty}$, \[\quad f_i(s)\tend{i\rightarrow\infty}{}f_\infty(s),\] where
		\[f_\infty(s)=\frac{-(1+\alpha s)+\sqrt{1+2\alpha s +\alpha^2s^2-4\alpha d+ 4\alpha\beta d-4\alpha\beta s}}{2}.\]
		\item For all $s\in\intervallefo{d}{s_\infty}$, we have $f_\infty(s)>-1.$ \label{gluing:limite}
		\item \label{gluing:convergence s_i}The sequence $(s_i)_{i\geq1}$ is strictly decreasing and
		\[s_i\tend{i\rightarrow\infty}{}s_\infty. \]
		\item For all $i\geq 1$, we have $f_{i+1}(s_{i})<-1.$\label{gluing:sis}
	\end{enumerate}
\end{proposition}
\begin{proof} 
	We define the function $F$ on the set $\enstq{(s,x)\in\R^2}{d\leq s \leq \frac{1}{\alpha}, \ x>\alpha d -\alpha s-1}$ by the expression:
	\[F(s,x)=\frac{\alpha(-d+\beta d -\beta s - d x)}{1+x+\alpha s -\alpha d}.\]
	We have for all $ s>d$ and all $x>\alpha d -\alpha s-1,$ \[ \partial_x F(s,x)=\frac{\alpha(\beta-\alpha d)(s-d)}{(1+x+\alpha s -\alpha d)^2}>0.\] This shows that for all  $s>d$, the function $F(s,\cdot)$ is strictly increasing, and also strictly concave since the derivative is strictly decreasing.\\
	From these facts we can show by induction on $i$ the points \ref{gluing:continue} and \ref{gluing:decroissance} of Proposition~\ref{gluing:prop des fi}, together with the validity of the definition of $f_i$ and $s_i$, in Definition-Proposition~\ref{gluing:def:fonctions fi}. 
	
	$\bullet$ For $i=1$, the function $f_1$ is well-defined, $s_1$ is indeed the unique solution to $f_1(s)=-1$ and the point \ref{gluing:continue} is satisfied. Moreover, $f_2(s)$ is well-defined for $s\in \intervalleff{d}{s_1}$ by $f_2(s)=F(s,-\alpha s)$ and for all $s\in\intervalleoo{d}{s_1}$, we have \[F(s,-\alpha s)+\alpha s=\frac{\alpha(\beta-1)(d-s)}{1-\alpha d}<0,\]
	which proves that \ref{gluing:decroissance} holds for $i=1$.
	
	$\bullet$ By induction, if $f_i$ and $s_i$ are defined up to some $i\geq1$ and satisfy \ref{gluing:continue}, then one can verify that for all $s\in\intervalleff{d}{s_{i}}$, the function $f_{i+1}$ is well-defined by the formula:
	\[f_{i+1}(s)=F(s,f_{i}(s)).\]
	From the monotonicity of $F(s,\cdot)$ and $f_i$, this function is continuous and strictly decreasing. One can check that $F(d,x)=-\alpha d$ for any $x>-1$ so $f_{i+1}$ satisfies \ref{gluing:continue}. Then, if $i=1$, the initialisation already gives us that \ref{gluing:decroissance} holds. Otherwise, if $i\geq2$, then using the induction hypothesis, for all $ s\in\intervalleof{d}{s_{i-1}}$ we have $f_{i}(s)<f_{i-1}(s)$. Using that $F(s,\cdot)$ is strictly increasing for $s>d$ we get that for all $s\in\intervalleof{d}{s_{i}}, \ f_{i+1}(s)<f_{i}(s)$, and so \ref{gluing:decroissance} holds. Since $f_{i+1}$ is continuous and strictly decreasing and that $f_{i+1}(d)>-1$ and $f_{i+1}(s_i)<f_i(s_i)=-1$, then $s_{i+1}$ is well-defined. This finishes our proof by induction.
	
	Let us study at fixed $s>d$ the equation $F(s,x)=x$. We get the following second order equation:
	\[x^2+x(1+\alpha s)+(\alpha d-\alpha\beta d+\alpha\beta s)=0,\]
	for which the discriminant is $\Delta_s=1+2\alpha s +\alpha^2s^2-4\alpha d +4\alpha\beta d-4\alpha\beta s$. We can evaluate this quantity at $d$ and at $\frac{1}{\alpha}$. We get
	\[\Delta_{d}=(\alpha d -1)^2>0 \qquad \text{and} \qquad \Delta_{1/\alpha}=4(\beta-1)(\alpha d -1)<0.\]
	We can check that it vanishes exactly at $s=s_\infty$ so that in the end, $\Delta_s$ is strictly positive on $\intervallefo{d}{s_\infty}$, null at $s_\infty$ and strictly negative on $\intervalleof{s_\infty}{\frac{1}{\alpha}}$. Hence, the function $F(s,\cdot)$ has $2$ (resp. $1$, resp. $0$) fixed points on the corresponding intervals.
	
	The convergence \ref{gluing:convergence} is a consequence of the fact that for $s\in\intervallefo{d}{s_\infty}$, the function $F(s,\cdot)$ is strictly increasing and concave, has exactly two fixed points and that the initial value $f_1(s)$ is greater than the smallest fixed point. We then have a convergence of the sequence $f_i(s)$ towards the greatest fixed point of $F(s,\cdot)$, the value of which can be computed using the equation above. The property of the limit \ref{gluing:limite} can be checked by proving that for all $ s\in\intervalleff{d}{s_\infty},$ we have $f_\infty(s)\geq f_\infty(s_\infty)=\sqrt{(\beta-1)(\beta-\alpha d)}-\beta>-1$.
	
	Let us prove the point \ref{gluing:convergence s_i}. According to property \ref{gluing:decroissance}, we have $f_i(s_{i+1})>f_{i+1}(s_{i+1})=-1$, and since $f_i$ is decreasing, we get $s_{i+1}<s_i$. Hence the sequence $(s_i)_{i\geq 1}$ is strictly decreasing, bounded below by $d$, so it converges. Now let $s>s_\infty$. If the sequence $(f_i(s))_{i\geq1}$ was well-defined for all $i\geq 1$, then for all $i\geq 1$ we would have $f_i(s)>-1$, so it would be decreasing, bounded below, hence it would have a limit, which would be a fixed point of $F(s,\cdot)$. It is impossible since $F(s,\cdot)$ has no fixed point, so the sequence is not well-defined for all $i\geq 1$ and so for $i$ large enough, $s>s_i$. We conclude that $\lim_{i\rightarrow\infty}s_i\leq s_\infty$. If we had $\lim_{i\rightarrow\infty}s_i<s_\infty$, then it would contradict the property \ref{gluing:limite}. In the end, $\lim_{i\rightarrow\infty}s_i= s_\infty$.
	
	The last property \ref{gluing:sis} follows from property \ref{gluing:decroissance}. Indeed, we have $f_{i+1}(s_{i})<f_{i}(s_{i})=-1$.
\end{proof}

\subsubsection{Construction of the coverings}
Let us provide a rigorous proof of our upper-bound, which follows the heuristics that we derived in the beginning of the section. Here, we distinguish two types of negligible functions, $\petiton$ and $\petitoe$. A function denoted $\petiton$ (resp. $\petitoe$) is negligible as $n\rightarrow\infty$ (resp. as $\epsilon\rightarrow 0$) and does not depend on $\epsilon$ (resp. on $n$).  
\begin{proposition}\label{gluing:recouvrements}
	Fix $ i\geq 1$ and $s\leq s_{i-1}$. For all $\epsilon>0$, we can construct simultaneously for all $n\geq 1$ a set of balls $R_{n,i}^{s,\epsilon}$, such that the following holds.
	\begin{enumerate}
		\item Almost surely, for $n$ large enough, $R_{n,i}^{s,\epsilon}$ covers $\cT(\bb_n)$.
		\item We have \[\Ec{\Vol_s(R_{n,i}^{s,\epsilon})}\leq n^{f_i(s)+\petiton+\petitoe}.\] 
		\item The diameter of the balls used are such that $\underset{B\in R_{n,i}^{s,\epsilon}}{\max}\diam B\tend{n\rightarrow\infty}{}0$.
	\end{enumerate} 
\end{proposition}
We will define the set of balls $R_{n,i}^{s,\epsilon}$ over the block $\bb_n$ and its descendants in an algorithmic way, and each step of the algorithm only depends on the gluings that happen after time $n$.
The proof of the upper-bound will directly follow from this proposition. Let us first state an elementary result, the proof of which is left to the reader. Note that we allow the function $\petitoe$ to be infinite for large values of $\epsilon$.
\begin{lemma}\label{gluing:minilemme}
	Let $\xi\geq -1$. Then for all $\gamma>1$, we have:
	\[\Ec{\sum_{k=n+1}^{n^\gamma} \ind{\bb_k\rightarrow \bb_n}k^{\xi+\petiton+\petitoe}}\leq  n^{-\beta+\gamma(\xi+1)+\petiton+\petitoe}.\]
\end{lemma}


\begin{proof}[Proof of Proposition~\ref{gluing:recouvrements}]
	Let $s>0$ and $\epsilon>0$. We prove the proposition by induction on $i$. The first set of balls that we build is the following: for each block $\bb_n$, we cover the block with a ball of radius $n^{-\alpha+\epsilon}$, centred on the point $\brho_n$. We write:
	\[R_{n,1}^{s,\epsilon}=\{\Ball (\brho_n,n^{-\alpha+\epsilon})\}.\]
	According to \eqref{gluing:hauteur sous arbres}, there exists a random $N$ such that for all $n\geq N$, the set $R_{n,1}^{s,\epsilon}$ covers $\cT(\bb_n)$. The diameter of the ball of $R_{n,1}^{s,\epsilon}$ tend to $0$ as $n\rightarrow\infty$. Besides we have, 
	\[\Ec{\Vol_s(R_{n,1}^{s,\epsilon})}\leq (2n)^{-\alpha s+\epsilon s}=n^{f_1(s)+\petiton+\petitoe}.\] 
	The property is thus proved for $i=1$.
	
	Let $i\geq 1$ and $s<s_{i}$. Let us construct $\left(R_{n,i+1}^{s,\epsilon}\right)_{n\geq1}$, using the previous step $i$. We set $\gamma_{i+1}(s)>1$ a positive real number that we will choose later, and $\epsilon>0$. We define $R_{n,i+1}^{s,\epsilon}$ as follows: it is the union over all the blocks $\bb_k$ for $k<n^{\gamma_i(s)}$ that are grafted on the block $\bb_n$, of their covering $R_{k,i}^{s,\epsilon}$ of the preceding step, together with the union of a deterministic set of balls that we define hereafter. 
	
	We want to cover $\bb_n$ with balls of radius $n^{-\alpha\gamma_{i+1}(s)}$, which is equivalent to covering $\bB_n$ with balls of radius $\lambda_n^{-1} n^{-\alpha\gamma_{i+1}(s)}$. 
	Under Hypothesis~\ref{gluing:hypothese d}, for any $d\geq0$, using Lemma~\ref{gluing:majoration M_r} and Lemma~\ref{gluing:lem:random element of prb} in Appendix~\ref{gluing:subsec:decomposition in fragments}, we can a.s.\ find a random collection $(x_m)_{1\leq m\leq M_{r}(\bB_n)}$ of points of $\bB_n$ such that the balls centred on those points with radius $r:=\lambda_n^{-1} n^{-\alpha\gamma_{i+1}(s)}$ cover $\bB_n$, and such that $M_{r}(\bB_n)\leq N_{r/4}(\bB_n)$, where $N_r(\bB)$ is the minimal number of balls of radius $r$ needed to cover $\bB$.
	
	From the assumption on the sequence $(\lambda_n)$, we have $r\geq n^{-\alpha\gamma_{i+1}(s)+\alpha+\petiton}$. Since $N_r(\bB_n)$ is decreasing in $r$, using Hypothesis~\ref{gluing:hypothese d}\ref{gluing:nombre mininimal de boules} we get that \[\Ec{N_{r/4}(\bB_n)}\leq n^{-\alpha d+\alpha\gamma_{i+1}(s)d +\petiton}.\]
	In the end, \[R_{n,i+1}^{s,\epsilon}:=\left(\bigcup_{k\leq n^{\gamma_{i+1}(s)}:\bb_k\rightarrow \bb_n}R_{k,i}^{s,\epsilon}\right)\cup \enstq{\Ball \left(x_m,n^{-\alpha\gamma_{i+1}(s)+\epsilon}\right)}{1 \leq m \leq M_r(\bB_n)}.\]
	Now we compute the expectation of the $s$-volume of these sets of balls.
	\begin{align*}
	&\mathrel{} \Ec{\Vol_s(R_{n,i+1}^{s,\epsilon})}\\
	&=\Ec{\sum_{k=n+1}^{n^{\gamma_{i+1}(s)}} \ind{\bb_k\rightarrow \bb_n}\Vol_s(R_{k,i}^{s,\epsilon})}+\Ec{M_r(\bB_n)} (2n^{(-\alpha\gamma_{i+1}(s)+\epsilon)})^s\\
	&\leq\Ec{\sum_{k=n+1}^{n^{\gamma_{i+1}(s)}} \ind{\bb_k\rightarrow \bb_n}\Ec{\Vol_s(R_{k,i}^{s,\epsilon})}}+\Ec{N_{r/4}(\bB_n)} (2n^{(-\alpha\gamma_{i+1}(s)+\epsilon)})^s\\
	&\leq \Ec{\sum_{k=n+1}^{n^{\gamma_{i+1}(s)}} \ind{\bb_k\rightarrow \bb_n}k^{f_i(s)+\petiton+\petitoe}}+n^{-\alpha d +\alpha d \gamma_{i+1}(s)+(-\alpha\gamma_{i+1}(s)+\epsilon)s+\petiton+\petitoe}\\
	&\leq n^{-\beta+\gamma_{i+1}(s)(f_i(s)+1)+\petiton+\petitoe}+n^{-\alpha d + \alpha d \gamma_{i+1}(s)-\alpha\gamma_{i+1}(s) s+\petiton+\petitoe},
	\end{align*}
	where in the last line we used Lemma~\ref{gluing:minilemme} which applies because $s\leq s_i$, hence $f_i(s)\geq -1$. We then take $\gamma_{i+1}(s):=\frac{\beta-\alpha d}{f_i(s)+1-\alpha d +\alpha s}>1$, which yields:\[-\beta+\gamma_{i+1}(s)(f_i(s)+1)=-\alpha d+\alpha\gamma_{i+1}(s)d-\alpha\gamma_{i+1}(s) s=f_{i+1}(s).\]
	We then have,
	\[\Ec{\Vol_s(R_{n,i+1}^{s,\epsilon})}\leq n^{f_{i+1}(s)+\petiton+\petitoe}.\]
	We can check that $\max_{B\in R_{n,i+1}^{s,\epsilon}}\diam B\tend{n\rightarrow\infty}{}0$, and that almost surely, for $n$ large enough, the collections of balls $R_{n,i+1}^{s,\epsilon}$ are indeed coverings of $\cT(\bb_n)$ thanks again to \eqref{gluing:hauteur sous arbres}. This finishes the proof.
\end{proof}
We can now prove the main proposition of this section.
\begin{proof}[Proof of Proposition~\ref{gluing:dimension}]
	Let $i\geq1$. For $\epsilon>0$ small enough, we use Proposition~\ref{gluing:recouvrements} to get a set of balls $(R_{n,i+1}^{s_{i},\epsilon})_{n\geq 1}$, which satisfies:
	\[\Ec{\Vol_{s_{i}}(R_{n,i+1}^{s_i,\epsilon})}\leq n^{f_{i+1}(s_{i})+\petiton+\petitoe}.\]
	From Proposition~\ref{gluing:prop des fi}\ref{gluing:sis}, we have $f_{i+1}(s_{i})<-1$, so we can choose $\epsilon$ small enough such that the exponent is eventually smaller than $\frac{f_{i+1}(s_{i})-1}{2}<-1$ as $n\rightarrow\infty$. Then, for $N\geq 1$, we set $R_N=\bigcup_{n\geq N}R_{n,i+1}^{s_{i},\epsilon}$.
	According to Proposition~\ref{gluing:recouvrements}, the set of balls $R_{n,i+1}^{s_{i},\epsilon}$ is a covering of $\cT(\bb_n)$ for $n$ large enough and so $R_N$ is a covering of $\cL$, for all $N$.
	Since for any $\delta>0$, we may choose $N$ large enough so that $\max_{B\in R_N}\diam B <\delta$, we get:
	\[\cH_{s_{i}}(\cL)=\lim_{\delta\rightarrow 0}\cH_{s_{i}}^\delta(\cL)\leq \limsup_{N\rightarrow\infty}\Vol_{s_{i}}(R_N)\leq \Vol_{s_{i}}(R_1)<+\infty, \quad \text{a.s.}\]
	since \[\Ec{\Vol_{s_{i}}(R_1)}\leq \sum_{n=1}^\infty n^{\frac{f_{i+1}(s_{i})-1}{2}+\petiton}<+\infty.\]
	This shows that the Hausdorff dimension of $\cL$ satisfies $\mathrm{dim_H}(\cL)\leq s_{i}$, almost surely. In the end, since the sequence $(s_i)_{i\geq 1}$ tends to $s_\infty$, we conclude that almost surely, \[\mathrm{dim_H}(\cL)\leq s_\infty=\frac{2\beta-1-2\sqrt{(\beta-1)(\beta-\alpha d)}}{\alpha}.\qedhere\]
	
\end{proof}
\section{Lower-bounds for the $(\alpha,\beta)$-model}\label{gluing:sec: lower bounds}
In this section we compute lower-bounds on the Hausdorff dimension of the set $\cL$. We do that by constructing Borel measures on $\cL$ that satisfy the assumptions of Frostman's lemma (Lemma~\ref{gluing:frostman lemma} in Appendix~\ref{gluing:subsec:hausdorff dimension}). 
In the case where $\beta\leq1$ we use the natural measure $\bar{\mu}$ on $\cT$ which arises as the limit of the normalised weight measures on $\cT_n$ (see Proposition~\ref{gluing:convergence of the measure mu}). The case $\beta>1$ is a bit more technical because the natural measure $\bar{\mu}$ is not concentrated on $\cL$, so we have to construct another measure $\pi$, that we define as the subsequential limit of some well-chosen sequence of probability measures on $\cT$. 
\subsection{Case $\beta\leq1$ and use of the measure $\bar{\mu}$}\label{gluing:subsec:lower bound measure mu}
In this subsection, we suppose that $\beta\leq 1$. Under the assumptions of Proposition~\ref{gluing:convergence of the measure mu}, the sequences of measures $\bar{\mu}_n$ almost surely converges weakly to a measure $\bar{\mu}$, which is concentrated on the set of leaves $\cL$. The existence of $\bar{\mu}$ will be useful for the proof of the next proposition. Recall from \eqref{gluing:hauteur point unif} the definition of the random variable $\cH$ and the fact that the assumptions on $\cH$ in the proposition are satisfied under Hypothesis~\ref{gluing:hypothese d} for any $d\geq0$.
\begin{proposition}\label{gluing:prop:minoration beta petit}
	Suppose that Hypothesis~\ref{gluing:condition beta petit} or Hypothesis~\ref{gluing:condition beta 1} is satisfied. Suppose also that $\Ec{\cH^2}<\infty$ and that $\Pp{\cH>0}>0$. Then the Hausdorff dimension of $\cL$ almost surely satisfies:
	\[\mathrm{dim_H}(\cL)\geq \frac{1}{\alpha}.\]
\end{proposition}
As we said earlier, the idea is to prove this lower bound on the dimension using Frostman's lemma: we will thus prove that almost surely, for $\bar{\mu}$-almost all leaves $x\in\cL$, we have an upper bound of the type \[\bar{\mu}(\Ball (x,r))\leq r^{1/\alpha-\epsilon},\] for $r$ sufficiently small, and for all $\epsilon$. An application of Lemma~\ref{gluing:frostman lemma} will then finish our proof. 

In order to prove this control on the masses of the balls, we will use two lemmas. The first one allows us to compare $\bar{\mu}(\Ball (x,r))$ with a quantity of the form $\bar{\mu}(\cT(\bb_n))$ for an appropriate $n$. The second one, Lemma~\ref{gluing:urne de polya}, provides a good control of the quantities $\bar{\mu}(\cT(\bb_n))$ for large $n$, such that the combination of the two will provide the upper bound that we want.
Let $\epsilon>0$. Recall from \eqref{gluing:def G epsilon} the definition of $G^\epsilon$.
\begin{lemma}\label{gluing:suite vers feuille 2}
	Set $n_0=2$ and $n_{k+1}=\lceil n_k^{1+\epsilon}\rceil $. Under the hypotheses of Proposition~\ref{gluing:prop:minoration beta petit}, almost surely for $\bar{\mu}$-almost every $x\in\cL$, for all $k$ large enough\footnote{The threshold depends on the realisation and on $x$.}, there exists $n \in \intervalleentier{n_k}{n_{k+1}}\cap G^\epsilon$ such that
	\[x\in\cT(\bb_{n}) \quad \text{and} \quad \dist(x,\brho_{n})\geq n^{-\alpha - 2\epsilon}.\]
\end{lemma}
\begin{proof}
	Note that in our setting, the hypothesis of Proposition~\ref{gluing:convergence of the measure mu} holds and so the random leaf $Y$ constructed in Section~\ref{gluing:height of a uniform point} is defined a.s. Also, according to Proposition~\ref{gluing:convergence of the measure mu}, conditionally on $(\cT,\bar{\mu})$, the point $Y$ has distribution $\bar{\mu}$. So it suffices to prove that the lemma holds for the the random leaf $Y$. We recall
	\[\dist(Y_n,Y_m)\overset{\mathrm{law}}{=}\sum_{k=m+1}^n \lambda_k \cH_k \ind{U_{k}\leq \frac{w_{k}}{W_{k}}} \quad \text{and} \quad Y=\lim_{n\rightarrow \infty} Y_n.\]
	Let us introduce a constant $c>0$ and set
	\[p:=\Pp{\cH>c}.\] 
	
	For $\beta<1$, thanks to our assumptions, we can fix $c$ such that $p$ is non-zero. We then have:
	\begin{align*}
	\Pp{\forall i \in \intervalleentier{n_k}{n_{k+1}}\cap G^\epsilon, \ U_i> \frac{w_i}{W_i} \ \text{or} \ \cH_i<c}
	&=\underset{i\in G^\epsilon}{\prod_{i={n}_k}^{n_k^{1+\epsilon}}}\left(1-p\frac{w_i}{W_i}\right)\\
	&= \exp\left({\underset{i\in G^\epsilon}{\sum_{i= n_k}^{n_k^{1+\epsilon}}}\log\left(1-p\frac{w_i}{W_i}\right)}\right)\\
	&\leq \exp\left({-p\underset{i\in G^\epsilon}{\sum_{i={n}_k}^{{n}_k^{1+\epsilon}}}\frac{w_i}{W_i}}\right)\\
	&\underset{\text{Lem.} \ref{gluing:divergence en log}}{\leq} \exp\left(-pC_\epsilon \log({n}_k)\right).
	\end{align*}
	To write the last line we use Lemma~\ref{gluing:divergence en log} in the Appendix and we can see that the last display is summable over $k$. 
	
	For the case $\beta=1$, Hypothesis~\ref{gluing:condition beta 1} allows us to write 
	\[\Pp{\forall i \in \intervalleentier{n_k}{n_{k+1}}\cap G^\epsilon, \ U_i> \frac{w_i}{W_i} \ \text{or} \ \cH_i<c}\leq \exp\left(-p f(k)\log\log\log({n}_k)\right),\]
	with a function $f(k)$ tending to infinity. Since $n_k\geq 2^{(1+\epsilon)^k}$, then $\log\log\log({n}_k)\geq (1+\petito{1}) \log k$ and the last display is also summable in $k$.
	In both cases, an application of the Borel-Cantelli lemma shows that we have almost surely, for $k$ large enough, 
	\[\exists n \in \intervalleentier{n_k}{n_{k+1}}\cap G^\epsilon, \qquad U_n\leq\frac{w_n}{W_n} \quad \text{and} \quad \cH_n\geq c.\]
	Since $n\in G^\epsilon$, we have $\lambda_n\geq n^{-\alpha-\epsilon}$. Combined with the fact that $\cH_n\geq c$ we get \[\dist(\brho_n,Y)\geq \lambda_n \cH_n \geq c n^{-\alpha-\epsilon}\geq n^{-\alpha-2\epsilon}, \text{ for $n$ (or equivalenly $k$) large enough}.\qedhere\]
\end{proof}
\begin{proof}[Proof of Proposition~\ref{gluing:prop:minoration beta petit}]
	Let $\epsilon>0$. Let us fix a realisation of $\cT$ and a leaf $x\in\cL$ such that the conclusions of Lemma~\ref{gluing:suite vers feuille 2} and Lemma~\ref{gluing:urne de polya} hold. Note that thanks to Hypothesis~\ref{gluing:condition beta petit} or Hypothesis~\ref{gluing:condition beta 1}, the condition of application of Lemma~\ref{gluing:urne de polya} are fulfilled. From the definition of $n_{k+1}$ we have $n_k^{1+\epsilon}<n_{k+1}\leq n_k^{1+\epsilon}+1$ and so $n_{k+1}\underset{k\rightarrow\infty}{=}n_k^{1+\epsilon+\petito{1}}$.
	We know from Lemma~\ref{gluing:suite vers feuille 2} that for all $k$ large enough, there exists $n\in\intervalleentier{n_k}{n_{k+1}}$ such that $x\in\cT(\bb_n)$ and 
	\[\dist(\brho_{n},x)\geq n^{-\alpha-2\epsilon}\geq n_{k+1}^{-\alpha-2\epsilon}\geq n_k^{(1+\epsilon+\petito{1})(-\alpha-2\epsilon)}.\] 
	So if we take $k$ large enough and $r\in\intervallefo{n_{k+2}^{-\alpha-2\epsilon}}{ n_{k+1}^{-\alpha-2\epsilon}}$, then
	\begin{align*}
	\bar{\mu}(\Ball(x,r))\leq \bar{\mu}(\cT(\bb_{n})) \underset{\text{Lem. } \ref{gluing:urne de polya}}{\leq} n^{-1+\epsilon} \leq n_k^{-1+\epsilon}&= n_{k+2}^{\frac{-1+\epsilon}{(1+\epsilon)^2}+\petito{1}}\\
	&\leq \left(r^{\frac{-1}{\alpha+2\epsilon}}\right)^{\frac{-1+\epsilon}{1+\epsilon}+\petito{1}}\\
	&\leq r^{\frac{1}{\alpha}+g(\epsilon)+\petito{1}},
	\end{align*}
	with a function $g$ tending to $0$ as $\epsilon\rightarrow 0$.
	
	Since the last display is true almost surely for all $r$ sufficiently small, we use Lemma~\ref{gluing:frostman lemma} (Frostman's lemma) to deduce that the Hausdorff dimension of $\cL$ is a.s.\ larger than $\frac{1}{\alpha}+g(\epsilon)$. Taking $\epsilon\rightarrow 0$ we get that almost surely, \[\mathrm{dim_H}(\cL)\geq \frac{1}{\alpha}.\qedhere\]
\end{proof}
\subsection{Case $\beta>1$ and construction of measures on the leaves}\label{gluing:subsec:construction of measures on the leaves}
The following section is devoted to prove the following proposition.
\begin{proposition}\label{gluing:prop:minoration dimension hausdorff}
	Suppose that Hypothesis~\ref{gluing:condition beta grand} is satisfied and that the block $\bB$ satisfies Hypothesis~\ref{gluing:hypothese d} for some $d\geq 0$. Then the Hausdorff dimension of $\cL$ almost surely satisfies:
	\[ \begin{aligned}
	\mathrm{dim_H}(\cL)&\geq \frac{2\beta-1-2\sqrt{(\beta-1)(\beta-\alpha d)}}{\alpha}, \quad \text{if } \alpha<\frac{1}{d},\\
	&\geq \frac{1}{\alpha} \quad\text{otherwise.}\end{aligned}\]
\end{proposition}
In the case $\beta>1$, we cannot use the natural measure $\bar{\mu}$ to get a good lower bound on the Hausdorff dimension of $\cL$ since, as stated in Proposition~\ref{gluing:convergence of the measure mu}, the measure $\bar{\mu}$ does not charge the leaves. So the goal of this subsection is  to artificially construct a probability measure concentrated on the leaves that will give us, using Frostman's lemma, the appropriate lower bound on the Hausdorff dimension, that is, the one matching with the upper bound derived in Section~\ref{gluing:sec:upper bounds}. The measure will be obtained as a sub-sequential limit of a sequence of measures concentrated on the blocks, and will only charge a strict subset of $\cL$. 

First, let us fix some notation. Recall the definition of $G_n^\epsilon$ in \eqref{gluing:def G epsilon n}.It follows from \ref{gluing:condition beta grand} that there exists a function $h(n)$ tending to $0$ such that $\#G_n^{h(n)}=n^{1+\petito{1}}$. We choose such a function $h$ and let 
\begin{equation}\label{gluing:def Gn hn}G_n:=G_n^{h(n)}.
\end{equation}
We will also use an increasing sequence of positive integers $(n_k)_{k\geq 0}$, such that for all $k\geq 0$, we have $n_{k+1}=\lceil n_k^{\parametre} \rceil $, with a fixed $\parametre>1$, that we will optimise later. Also we suppose $n_0$ to be very large, with conditions that we will make explicit in what follows. For all $n\geq 1$, we set \[\overline{\bb}_n:=\enstq{x\in \bb_n}{\bd_n(\brho_n,x)>\frac{\haut(\bb_n)}{2}},\]
the "upper-half" of block $\bb_n$. For technical reasons, we will only keep in our construction the blocks that behave reasonably well, see the forthcoming property \eqref{gluing:property P}, introduced in Section~\ref{gluing:sec:mass estimations}. We recursively define some random sets of integers. Let
\begin{equation}\label{gluing:eq:definition tildeGn_0}
\widetilde{G}_{n_0}=\enstq{n\in G_{n_0}}{\bB_n\text{ satisfies property \eqref{gluing:property P}}},
\end{equation}
and for $k\geq 0$,
\begin{equation}\label{gluing:eq:definition tildeGn_k+1}
\widetilde{G}_{n_{k+1}}=\enstq{n\in G_{n_{k+1}}}{\bB_n\text{ satisfies property \eqref{gluing:property P}};\ \exists i \in \widetilde{G}_{n_{k}},\ \bb_n\rightarrow \bb_i, \ X_{n-1}\in \overline{\bb}_i }.
\end{equation}
We then define for all $k\geq 0$, 
\[\cB_{k}=\bigcup_{n\in \widetilde{G}_{n_{k}}}\overline{\bb}_n.\]
In other words, $\cB_0$ is the union of all the upper-halves of the blocks $\bb_n$, for $n$ in $G_{n_0}$ for which $\bB_n$ behaves well, and $\cB_{k+1}$ is defined to be the union of all the upper-halves of the blocks of index  $n\in G_{n_{k+1}}$ that are grafted directly on $\cB_k$, and such that $\bB_n$ behaves well. Note that for the moment, $\cB_k$ can be empty.

We define the measure $\sum_{n\in\widetilde{G}_{n_k}}\bnu_n$ and refer to it as the \emph{mass measure}\footnote{It may appear more natural to define the mass measure as simply $\sum_{n=1}^{\infty}\nu_n$ which gives $\cT$ a finite mass. However, this definition would not have the property that for every set $S\subset \overline{\bb}_n$, we have $\abs{S}=\bnu_n(S)$. Indeed, when blocks can have an atom at their root, which is possible in the case $d=0$, the contribution of the mass added by the roots of future blocks should be counted in $\abs{S}$.}
on the $k$-th generation. To simplify notation we denote it $\abs{\, \cdot \, }$. We do not index it by $k$ since the index for which we consider it is always clear from the context.
We also define a sequence $(\pi_k)_{k\geq 0}$ of probability measures on $\cB_k$ by
\[\pi_k:=\frac{\abs{\, \cdot \cap \cB_k}}{\abs{\cB_k}},\]
the normalised mass measure on $\cB_k$.
Note that the sequence $(\pi_k)$ is only well-defined on the event where $\cB_k$ has non-zero mass for all $k$. In what follows we will ensure that it is the case for an event of strictly positive probability and only work conditionally on this event. Remark that, still conditionally on this event and on the event that $\cT$ is compact, which has probability $1$, the sequence $(\pi_k)_{k\geq 0}$ is a sequence of probability measures on a compact space, hence it admits at least one subsequential limit $\pi$ for the Lévy-Prokhorov distance. We can check using \cite[Lemma~17]{curien_random_2017}, which is essentially an application of the Portmanteau theorem, that $\pi$ is concentrated on $\bigcap_{k\geq 0 }\cT(\cB_k)\subset \cL$.

\subsubsection{Idea of the proof}\label{gluing:subsubsec:idea of the proof}
Let us briefly explain how the measure $\pi$ that we just constructed enables us to derive the appropriate lower bound for the Hausdorff dimension. We give the intuition for $\alpha<1/d$; the idea for $\alpha>1/d$ is very similar. We will be very rough for this sketch of proof and we keep the notation introduced above. Let us here forget that some blocks may not satisfy \eqref{gluing:property P}, and that we only deal with half-blocks.
\paragraph{Number of blocks in $\cB_k$}
Suppose that the number of blocks in $\cB_k$ evolves like a power of $n_k$, say $n_k^a$. Then the total weight of $\cB_k$ is $\abs{\cB_k}\approx n_k^a n_k^{-\beta}$, because all the blocks in $\cB_k$ have weight $\approx n_k^{-\beta}$. Since the probability that any block with index in $\intervalleentier{n_{k+1}}{2n_{k+1}}$ is grafted on $\cB_k$ is roughly $\abs{\cB_k}$, and since the number of blocks in $\cB_{k+1}$ is roughly $n_{k+1}^a$, we have \[ n_k^{\parametre a}\approx n_{k+1}^a\approx \abs{\cB_k}\cdot n_{k+1} \approx n_k^a n_k^{-\beta} n_k^\parametre.\]
Hence we have $a=\frac{\parametre-\beta}{\parametre -1}$, and so $\abs{\cB_k}\approx n_k^{\frac{\parametre(1-\beta)}{\parametre -1}}$.
\paragraph{Estimation on $\pi$} For each $k\geq 0$, the set $\cB_k$ is made of blocks of size $\approx n_k^{-\alpha}$. Let us suppose that the quantities of the form $\pi(\Ball(x,r))$ are well-approximated by $\pi_k(\Ball(x,r))$, whenever $r\in\intervalleff{n_{k+1}^{-\alpha}}{n_k^{-\alpha}}$. 
For $x$ close to $\cB_k$, such a ball typically intersects only one block of $\cB_k$, with weight roughly $n_k^{-\beta}$, and since the block is $d$-dimensional, the ball covers a proportion $\left(r/n_k^{-\alpha}\right)^d$ of this block. So $\abs{\Ball(x,r)\cap \cB_k}\approx n_k^{-\beta} \left(r/n_k^{-\alpha}\right)^d$, and 
\begin{align*}\pi(\Ball(x,r))\approx \pi_k(\Ball(x,r))= \frac{\abs{\Ball(x,r)\cap \cB_k}}{\abs{\cB_k}}&\approx r^dn_k^{-\beta +\alpha d -\parametre\frac{\beta-1}{\parametre-1}}\\
&\approx r^dn_k^{\frac{1}{\parametre -1}(\parametre(\alpha d -1) +\beta -\alpha d)}.
\end{align*}
Then, if $\alpha<1/d$, the last exponent is negative for $\parametre$ large enough. For such $\parametre$, using $r>n_{k+1}^{-\alpha}\approx n_k^{-\alpha\parametre}$ yields
\begin{equation}\label{gluing:heuristique lower bound}\pi(\Ball(x,r))\leq r^{d-\frac{1}{\alpha\parametre(\parametre-1)}(\parametre(\alpha d-1) +\beta -\alpha d )}.
\end{equation}
\paragraph{Optimisation}
We then choose $\parametre$ such that the exponent $d-\frac{\parametre(\alpha d-1) -\alpha d+\beta}{\alpha\parametre(\parametre-1)}$ is maximal. We get the value $\frac{2\beta-1-2\sqrt{(\beta-1)(\beta-\alpha d)}}{\alpha}$ which matches our upper-bound.

\paragraph{Plan of the proof}
Our goal is now to make those heuristics rigorous. First we will make some precise estimation on how the mass of the blocks with indices in $\widetilde{G}_{n_k}$ is spread on subsets of $\cT_{n_k-1}$. Then we will decompose each block of each $\cB_k$ into subsets that we call fragments for which our preceding estimation holds. After that, we use this decomposition to control the behaviour of the $(\pi_k)$, and also how the measures $\pi_k$ can approximate the limiting measure $\pi$. At the end we conclude by optimising on the parameters.

We will distinguish the two cases $d=0$ and $d>0$ and mostly work on the latter. We then explain quickly how the proof can be adapted to $d=0$, in which fewer technicalities are involved.
\subsubsection{Mass estimations}\label{gluing:sec:mass estimations}
Before proving our main proposition, we have to state some technical lemmas that will allow us to control how regularly the mass of $\cB_{k+1}$ is spread on $\cB_k$. Let us now define the property \eqref{gluing:property P}, in a different way whether $d=0$ or $d>0$. Let $C>0$ be a positive number and remind the definition of \eqref{gluing:controle boules} in Hypothesis~\ref{gluing:hypothese d}\ref{gluing:dimension d}. For $d>0$ we say that a pointed compact metric space endowed with a probability measure $(\bb,\bd,\brho,\bnu)$ satisfies \eqref{gluing:property P} iff
\begin{equation} \label{gluing:property P}\tag{$P_d$}
\left\lbrace\begin{aligned}
&C^{-1}\leq \haut(\bb)\leq C,\\
&\bnu\left(\enstq{x\in \bb}{\bd(\brho,x)\geq \frac{\haut(\bb)}{2}}\right)\geq C^{-1},\\
&(\bb,\bd,\brho,\bnu) \text{ satisfies \eqref{gluing:controle boules}} \ \text{with}\ r_0=C^{-1}.\\
\end{aligned}\right.
\end{equation}
For $d=0$ we say that $(\bb,\bd,\brho,\bnu)$ satisfies \eqref{gluing:property P0} iff $\bb$ is finite and 
\begin{equation} \label{gluing:property P0}\tag{$P_0$}
\left\lbrace\begin{aligned}
&C^{-1}\leq \haut(\bb)\leq C,\\
&\#\bb\leq C, \\
& \forall x\in\bb, \ \nu\left(\{x\}\right)\geq C^{-1}.
\end{aligned}\right.
\end{equation}

In any case, under Hypothesis~\ref{gluing:hypothese d}\ref{gluing:dimension d}, for any $d\geq 0$, we can choose $C$ such that the underlying block $(\bB,\bD,\bRho,\bNu)$ satisfies \eqref{gluing:property P} with a positive probability $p>0$.
\[p=\Pp{\bB \text{ satisfies }\eqref{gluing:property P}}>0.\] From now on we fix such a constant $C$. We also set $\bM:=\bNu\left(\enstq{x\in \bB}{\bD(\bRho,x)\geq \frac{\haut(\bB)}{2}}\right)$, and $m:=\Ecsq{\bM}{\bB \text{ satisfies } \eqref{gluing:property P}}$. We also denote by $\vM$ a random variable with the law of $\bM$ conditional on the event $\{\bB \text{ satisfies } \eqref{gluing:property P}\}$,
\begin{lemma}\label{gluing:mass estimations} Let $S$ be a subset of some $\bb_i$ with $i\leq n_{k}-1$,
	measurable with respect to $\cF_{n_k-1}$, the $\sigma$-field generated by the blocks and the gluings up to time $n_k-1$. Let $\chi(S)$ be the total mass of the union of the sets $\enstq{\overline{\bb}_n}{n\in G_{n_k}, \ \bb_n\rightarrow \bb_i,\ X_{n-1}\in S,\  \bB_n \ \text{satifies} \ \eqref{gluing:property P}}$, namely, the total mass of the half-blocks that are grafted on $S$ with index in $G_{n_k}$, and such that the corresponding blocks satisfy property \eqref{gluing:property P}.
	Then for all $x\in \intervalleff{0}{1}$,
	\[\Ppsq{\big|\chi(S)-a_k\abs{S}\big|>x a_k \abs{S}}{\cF_{n_k-1}}\leq 2\exp\left(-x^2 n_k^{1+\petito{1}}\abs{S}\right),\]
	where $\displaystyle a_k:=pm\sum_{i\in G_{n_k}}\frac{w_i}{W_{i-1}}$, is such that $\Ecsq{\chi(S)}{\cF_{n_k-1}}=a_k \abs{S}$ and the function $\petito{1}$ in the right-hand side does not depend on $x$.
\end{lemma}
This lemma roughly states that, for every subset $S\subset\bb_i$ for $i\in G_{n_k}$, if the subset has enough mass to attract a substantial number of the blocks coming between time $n_k$ and $2n_k$ then we have a good control on how the mass of $\cB_k$ grafted on $S$ can deviate from its expected value. 

\begin{proof}
	First we write $\chi(S)$ as:
	\[\chi(S)=\sum_{i\in G_{n_k}} \ind{U_i\leq\frac{\abs{S}}{W_{i-1}}}\ind{\bB_i \text{ satisfies \eqref{gluing:property P}}} \mathsf{M}_i w_i,\]
	where the $(U_i)$ are independent uniform variables on $\intervalleff{0}{1}$, independent of everything else. 
	Then we can compute
	\begin{align*}
	\Ecsq{\chi(S)}{\cF_{n_k-1}}&=\sum_{i\in G_{n_k}} \frac{\abs{S}}{W_{i-1}}p w_i \cdot \Ecsq{\bM_i}{\bB_i \text{ satisfies } \eqref{gluing:property P} }\\
	&=\abs{S}\cdot p m\left(\sum_{i\in G_{n_k}} \frac{w_i}{W_{i-1}}\right)=\abs{S}\cdot a_k.
	\end{align*}
	Let us bound the exponential moments of $\chi(S)$:
	\begin{align*}
	\Ecsq{\exp(\theta\chi(S))}{\cF_{n_k-1}}&=\prod_{i\in G_{n_k}} \left(\frac{p\abs{S}}{W_{i-1}}\cdot\Ecsq{e^{\theta w_i \bM_i}}{\bB_i \text{ satisfies } \eqref{gluing:property P}}+1-\frac{p\abs{S}}{W_{i-1}}\right)\\
	&=\prod_{i\in G_{n_k}} \left(1+\frac{p\abs{S}}{W_{i-1}}(\Ec{e^{\theta w_i \vM}}-1)\right)\\
	&\leq \exp\left(p\abs{S}\sum_{i\in G_{n_k}} \frac{1}{W_{i-1}}(\Ec{e^{\theta w_i \vM}}-1)\right),
	\end{align*}
	where we used the inequality $e^z\geq 1+z$, in the last line. Now,
	\begin{align*}
	&\mathrel{}\Ecsq{\exp(\theta(\chi(S)-a_k\abs{S}))}{\cF_{n_k-1}}\\
	&\leq \exp\left(p \abs{S}\sum_{i\in G_{n_k}} \frac{1}{W_{i-1}}(\Ec{e^{\theta w_i \vM}}-1)-\theta\abs{S}\sum_{i\in G_{n_k}} \frac{p mw_i}{W_{i-1}}\right)\\
	&\leq \exp\left(p\abs{S}\left(\sum_{i\in G_{n_k}} \frac{1}{W_{i-1}}(\Ec{e^{\theta w_i \vM}}-1-\theta m w_i)\right)\right)\\
	&\leq \exp\left(p\abs{S}\left(\sum_{i\in G_{n_k}} \frac{1}{W_{i-1}}c (\theta w_i)^2\right)\right).
	\end{align*}
	Here we used the fact that for $z\in\intervalleff{-1}{1}$, we have $ e^z\leq 1 + z+3z^2$ and so
	\begin{align*}
	\Ec{e^{z\vM}}-1-z \Ec{\vM} \leq \Ec{1+z\vM+3(z\vM)^2}-1-z \Ec{\vM}\leq 3 \Ec{\vM^2} z^2\leq c z^2,
	\end{align*}
	for $c$ a constant. Since we ask that $z\in\intervalleff{-1}{1}$, the computation above is valid if we restrict ourselves to $\abs{\theta} \leq (\sup_{i\in G_{n_k}}w_i)^{-1} = n_k^{\beta + \petito{1}}$. Note that we can use this inequality for negative values of $\theta$.
	Hence for $x\in\intervalleff{0}{1}$ we have
	\begin{align*}
	&{}\Ppsq{\abs{\chi(S)-a_k\abs{S}}>x a_k\abs{S}}{\cF_{n_k-1}}\\
	&\leq \Ppsq{\chi(S)-a_k\abs{S}>x a_k\abs{S}}{\cF_{n_k-1}}+\Ppsq{-(\chi(S)-a_k\abs{S})>x a_k\abs{S}}{\cF_{n_k-1}}\\
	&\leq \Ppsq{\exp\left(\theta(\chi(S)-a_k\abs{S})\right)>e^{\theta x a_k\abs{S}}}{\cF_{n_k-1}}+\Ppsq{\exp\left(-\theta\left(\chi(S)-a_k\abs{S}\right)\right)>e^{\theta x a_k\abs{S}}}{\cF_{n_k-1}}\\
	&\leq 2\exp\left(p\abs{S}\left(\sum_{i\in G_{n_k}} \frac{1}{W_{i-1}}c(\theta w_i)^2\right)-\theta x a_k\abs{S}\right)\\
	&=2\exp\left(p\abs{S}\left(\sum_{i\in G_{n_k}} \frac{1}{W_{i-1}}(c(\theta w_i)^2-\theta x m w_i)\right)\right).
	\end{align*}
	Taking $\theta = x n_k^{\beta-\epsilon}$ in the last inequality, which is possible for $n_k$ large enough, this gives
	\begin{align*}
	&\Ppsq{\big|\chi(S)-a_k\abs{S}\big|>x a_k\abs{S}}{\cF_{n_k-1}}\\
	&\leq 2\exp\left(p\abs{S} x^2\left(\sum_{i\in G_{n_k}} \frac{1}{W_{i-1}}(n_k^{\beta-\epsilon}w_i)(cn_k^{\beta-\epsilon}w_i-m)\right)\right).
	\end{align*}
	From our assumptions on the sequence $(w_n)$, we have $n_k^{\beta-\epsilon}w_i\rightarrow 0$ and hence $(cn_k^{\beta-\epsilon}w_i-m)$ is eventually smaller than $-\frac{m}{2}$, uniformly for $i\in G_{n_k}$. Also $\frac{1}{W_n}$ is always greater than $\frac{1}{W_\infty}$. Combining this with the last display we get, for $n_k$ large enough
	\begin{align*}
	&\Ppsq{\big|\chi(S)-a_k\abs{S}\big|>x a_k\abs{S}}{\cF_{n_k-1}}\\
	&\leq 2\exp\left(-x^2 \abs{S} \cdot \frac{pm}{2 W_\infty} \#G_{n_k} (\inf_{i\in G_{n_k}} w_i) n_k^{\beta-\epsilon}\right)\\
	&\leq 2\exp\left(-x^2 \abs{S} n_k^{1-\epsilon+\petito{1}}\right).
	\end{align*}
	Now for every $\epsilon>0$, this inequality is true for $n_k$ large enough, so this proves the lemma.
\end{proof}

Let us also state another technical lemma, the proof of which is in the Appendix \ref{gluing:subsec:computations}.
\begin{lemma}\label{gluing:lemme des ai}
	Suppose that Hypothesis~\ref{gluing:condition beta grand} is satisfied. We have, for the sequence $(a_i)$ defined in Lemma~\ref{gluing:mass estimations},
	\[\prod_{i=1}^k a_i=n_k^{\frac{\parametre(1-\beta)}{(\parametre-1)}+\petito{1}},\]
	where the $\petito{1}$ is considered when $k\rightarrow\infty$.
\end{lemma}
\subsubsection{Construction of fragments, case $d>0$}\label{gluing:construction of fragments}
\paragraph{Fragments of a random block}
Let us discuss how we can decompose a metric space into a partition of subsets that we call $r$-\emph{fragments}, all of them having a diameter of order $r$. Suppose that the random block $(\bB,\bD,\bRho,\bNu)$ comes with a sequence of random points $(X_n)_{n\geq1}$, which are i.i.d.\ with law $\bNu$, conditionally on $(\bB,\bD,\bRho,\bNu)$, and that this block satisfies Hypothesis~\ref{gluing:hypothese d}\ref{gluing:dimension d}, for some $d>0$. The following lemma ensures that in this setting, we can construct a partition $F(\bB,r)=(f^{(r)}_i)_{1\leq i\leq N}$ of $r$-fragments of $(\bB,\bD,\bRho,\bNu)$, which have approximately equal diameter and measure. 
Recall the function $\varphi$ defined in Hypothesis~\ref{gluing:hypothese d}\ref{gluing:dimension d}, and the notation $N_r(\bB) $ for the minimal number of balls of radius $r$ needed to cover $\bB$.

\begin{lemma}\label{gluing:lem:decomposition in fragments}
	Suppose that $(\bB,\bD,\bRho,\bNu)$ satisfies Hypothesis~\ref{gluing:hypothese d}\ref{gluing:dimension d} for some $d>0$. For any $r\in\intervalleff{0}{1}$, we construct a finite partition of Borel subsets $(f^{(r)}_i)_{1\leq i\leq N}$ of the block $(\bB,\bD,\bRho,\bNu)$ in a deterministic way from $(\bB,\bD,\bRho,\bNu)$ and the sequence of random points $(X_n)_{n\geq1}$. 
	There exists two functions $\psi$ and $\phi$ defined on the interval $\intervalleff{0}{1}$, which tend to $0$ at $0$ such that the following holds almost surely on the event $\{(\bB,\bD,\bRho,\bNu) \text{ satisfies \eqref{gluing:controle boules}} \}$, for any $r_0>3r$,
	\begin{enumerate}
		\item\label{gluing:it:taille masse fragment} For all $1\leq i\leq N$,
		\[\diam f_i \leq 2r \quad \text{and} \quad \left(\frac{r}{4}\right)^{d+\varphi(r/4)} \leq \bNu(f_i) \leq r^{d-\varphi(r)}. \]
		\item\label{gluing:it:nb fragments intersectant une boule} For all $r'<r_0/3$, we have
		\[ \forall x\in\bB, \quad \#\enstq{1\leq i \leq N}{\Ball (x,r')\cap f_i\neq\emptyset} \leq (r\vee r')^{d+\psi(r\vee r')}\cdot r^{-d+\phi(r)}.\]
		\item\label{gluing:it:borne nb fragments} The (random) number $N$ of fragments satisfies \[N\leq N_{r/4}\left(\bB\right) \quad \text{and}\quad  N\leq \left(\frac{r}{4}\right)^{-d-\varphi(r/4)}.\]
	\end{enumerate}
\end{lemma}

In the paper, we use this construction on $(\bB_n,\bD_n, \bRho_n, \bNu_n)$, assuming that for all $n\geq1$, a sequence $(X_{n,j})_{j\geq 1}$ is defined on the same probability space and that this sequence is i.i.d.\ with law $\bNu_n$, conditionally on $(\bB_n,\bD_n, \bRho_n, \bNu_n)$. For any $n\geq 1$ and $r>0$, we denote $F(\bB_n,r)=\enstq{f^{(r)}_i}{1\leq i \leq N}$ the partition of $\bB_n$ into (random) $r$-fragments which is given by the lemma. The proof of Lemma~\ref{gluing:lem:decomposition in fragments} can be found in the Appendix.

\paragraph{Decomposition of $\bb_n$ into fragments}

Fix a parameter $\eta\in\intervalleoo{0}{\frac{1}{d}}$. We want to decompose every $\bb_n$, for $n\in \widetilde{G}_{n_k}$, in fragments of size approximately $n_{k+1}^{-\eta}$. For that, it is sufficient to use $F(\bB_n,\mathsf r_n)$ the decomposition of $\bB_n$ in $\mathsf{r}_n$-fragments with $\mathsf{r}_n=(\lambda_n^{-1}\cdot n_{k+1}^{-\eta})$, which is given by Lemma~\ref{gluing:lem:decomposition in fragments}. Let us emphasise that these fragments are constructed as subsets of $\bB_n$, but we consider them as subsets of $\bb_n$ in what follows, without changing notation.
We define the set $F_k$ is as the collection of all these fragments coming from every $\bb_n$ with $n\in \widetilde{G}_{n_k}$. We have of course 
\begin{align}
\bigcup_{f\in F_k} f = \bigcup_{n\in \widetilde{ G}_{n_k}}\bb_n.
\end{align}
In our construction, we decided to keep only the blocks that were sufficiently well-behaved with respect to some properties that will be useful now. Recall the definition of the random set $\widetilde{G}_{n_k}$ in equations \eqref{gluing:eq:definition tildeGn_0} and \eqref{gluing:eq:definition tildeGn_k+1}.
Remark that, from the definition of $G_{n_k}$, we have,
\[c_k:=\min_{n\in  G_{n_k}}\left(\lambda_n^{-1}\cdot n_{k+1}^{-\eta}\right)= n_k^{\alpha-\parametre\eta+\petito{1}} \text{\quad and \quad} C_k:=\max_{n\in G_{n_k}}\left(\lambda_n^{-1}\cdot n_{k+1}^{-\eta}\right) = n_k^{\alpha-\parametre\eta+\petito{1}}.\]
If $\parametre$ and $\eta$ are such that $\parametre>\frac{\alpha}{\eta}$, then the last exponent is strictly negative, and so we can take $n_0$ sufficiently large so that $C_k<C^{-1}/3$, for all $k\geq 0$.
For $n\in \widetilde{G}_{n_k}$, we know that $\bB_n$ satisfies \eqref{gluing:controle boules} with $r_0=C^{-1}$. Hence, for all $n\in \widetilde{ G}_{n_k}$, we have $3\mathsf{r}_n\leq r_0$, and so the conclusions of Lemma~\ref{gluing:lem:decomposition in fragments} hold simultaneously for all the decompositions $F(\bB_n,\mathsf r_n)$ for $n\in \widetilde{ G}_{n_k}$.

\paragraph{Control on the mass and number of fragments}Recall the function $h$ that we defined in \eqref{gluing:def Gn hn}, which tends to $0$ at infinity, and the function $\varphi$ specified in Hypothesis~\ref{gluing:hypothese d}\ref{gluing:dimension d}, which tends to $0$ at $0$.
Thanks to Lemma~\ref{gluing:lem:decomposition in fragments}, we get, for all $f\in F_k$ such that $f\subset\bb_n$, 
\begin{align*}\abs{f}=w_n\cdot \bNu_n(f)
&\underset{\eqref{gluing:def Gn hn},\ \text{Lem.}\ref{gluing:lem:decomposition in fragments}\ref{gluing:it:taille masse fragment}}{\geq} n_k^{-\beta-h(n_k)}\cdot (\lambda_n^{-1}\cdot n_{k+1}^{-\eta})^{d+\varphi(n_k^{-\alpha d + \alpha +\petito{1}})}\\
&\geq n_k^{-\beta-h(n_k)}\cdot c_k^{d+\varphi(C_k)}.
\end{align*}
Note that the last quantity is deterministic and only depends on $n_k$, and so almost surely, 
\begin{align}\label{gluing:eq:min masse fragment}
\min_{f\in F_k}\abs{f}&\geq n_k^{-\beta-h(n_k)}\cdot c_k^{d+\varphi(C_k)}\notag \\
&\geq n_k^{-\eta d \parametre +\alpha d-\beta + \petito{1}}= n_{k+1}^{-\eta d + \frac{1}{\parametre}(\alpha d - \beta)+\petito{1}}.
\end{align}
Note that a similar computation using upper-bounds instead of lower-bounds also yields, almost surely,
\begin{equation}\label{gluing:max poids fragment}\max_{f\in F_k}\abs{f}\leq n_k^{-\beta+\petito{1}}\cdot C_k^{d-\varphi(c_k)}\leq n_{k}^{-\eta d \parametre+\alpha d-\beta+\petito{1}},
\end{equation}
where the right-hand side is deterministic. Also, from Lemma~\ref{gluing:lem:decomposition in fragments}\ref{gluing:it:borne nb fragments}, we get that the number of fragments obtained from the block $\bb_n$ by that construction is bounded above by $\left(\mathsf{r}_n/4\right)^{-d-\varphi(\mathsf{r}_n/4)}$, with $\mathsf{r}_n=\lambda_n^{-1}\cdot n_{k+1}^{-\eta}=n_k^{\alpha-\parametre\eta+\petito{1}}$, and so at the end, the total number of fragments in $F_k$ is bounded above by a deterministic quantity which grows at most polynomially in $n_k$.

\subsubsection{Construction of fragments, case $d=0$}
In this case, we will consider the finite number of points of each block as a decomposition into fragments, hence we set $F_k=\enstq{\{x\}}{x\in \bb_n,\  n\in\widetilde{G}_{n_k}}.$ Note that
\[\forall f\in F_k, \ f\subset\bb_n, \quad \abs{f}=w_n\cdot \bNu_n(f)
\underset{\eqref{gluing:def Gn hn},\ \eqref{gluing:property P0}}{\geq} n_k^{-\beta-h(n_k)}\cdot C^{-1},\]
and so the equations \eqref{gluing:eq:min masse fragment} and \eqref{gluing:max poids fragment} are still valid when $d=0$, and also the number of fragments in $F_k$ grows linearly, hence polynomially in $n_k$.

\subsubsection{Using the mass estimations}Recall that we fixed a parameter $\eta\in\intervalleoo{0}{\frac{1}{d}}$. We let $0<\epsilon<(1-\eta d)$. If $\parametre$ and $\eta$ are such that $\parametre>\frac{\beta-\alpha d}{1-\eta d-\epsilon}$, then $-\eta d +\frac{1}{\parametre}(\alpha d -\beta)>-1+\epsilon$. And so we get from \eqref{gluing:eq:min masse fragment} that $\min_{f\in F_k}\abs{f}\geq  n_{k+1}^{-1+\epsilon+\petito{1}}$. We can apply the result of Lemma~\ref{gluing:mass estimations} for every fragment $f\in F_{k}$, with $x=n_{k+1}^{-\epsilon/4}$,
\begin{align*}
&\Ppsq{\big|\chi(f)-a_{k+1}\abs{f}\big|>n_{k+1}^{-\epsilon/4} a_{k+1} \abs{f}}{\cF_{n_{k+1}-1}}\\
&\leq 2\exp\left(-\left(n_{k+1} ^{-\epsilon/4}\right)^2  n_{k+1} ^{1+\petito{1}} \min_{f\in F_k}\abs{f}\right)\\
&\leq 2\exp\left(-n_{k+1} ^{-\epsilon/2}  n_{k+1} ^{1+\petito{1}} n_{k+1} ^{-1+\epsilon+\petito{1}}\right)\leq 2\exp\left(- n_{k+1} ^{\epsilon/4}\right) \quad \text{for $n_{k+1}$ large enough.}
\end{align*}
For that, again, we impose that $n_0$ is large enough such that the last display is true for all $k$ and for all $f$. Now we can sum this over all fragments,
\begin{align*}
&\Ec{\sum_{k=0}^\infty\sum_{f\in F_{k}}\Ppsq{\big|\chi(f)-a_{k+1}\abs{f}\big|>n_{k+1}^{-\epsilon/4} a_{k+1} \abs{f} }{\cF_{n_{k+1}-1}}}\\
&\leq \sum_{k=0}^\infty \Ec{\#F_{k}}\cdot 2\exp\left(-n_{k+1}^{\epsilon/4}\right) \underset{n_0\rightarrow\infty}{\longrightarrow}0,
\end{align*}
since $(\#F_{k})$ is almost surely bounded by a deterministic quantity which grows at most polynomially in $n_k$. 
The same is true for the $\cB_k$,
\[\Ec{\sum_{k=0}^\infty\Ppsq{\big|\chi(\cB_k)-a_{k+1}\abs{\cB_k}\big|>n_{k+1}^{-\epsilon/4} a_{k+1} \abs{\cB_k} }{\cF_{n_{k+1}-1}}}\underset{n_0\rightarrow\infty}{\longrightarrow}0.\]

In the rest of Section~\ref{gluing:subsec:construction of measures on the leaves}, we will fix $n_0$ large enough and work on the event of large probability $\cE$ on which we have, for all $k\geq 0$ and for all $f\in F_k$
\begin{align}\label{gluing:eq:bounds mass grafted on fragments}
\big|\chi(\cB_k)-a_{k+1}\abs{\cB_k}\big|\leq n_{k+1}^{-\epsilon/4} a_{k+1} \abs{\cB_k} \quad \text{and} \quad  \big|\chi(f)-a_{k+1}\abs{f}\big|\leq n_{k+1}^{-\epsilon/4} a_{k+1} \abs{f}.
\end{align} 
Remark that thanks to Section~\ref{gluing:monotonicity of Hausdorff dimension}, giving a lower bound of the Hausdorff dimension on a set of positive probability is enough to prove that the bound holds almost surely.
Note that this construction depends on the parameters $\eta$ and $\epsilon$ and $\parametre$. The parameters must satisfy \begin{equation}\eta\in\intervalleoo{0}{\frac{1}{d}}, \quad \epsilon\in\intervalleoo{0}{1-\eta d}, \quad\parametre>\max\left(\frac{\alpha}{\eta},\frac{\beta-\alpha d}{1-\eta d-\epsilon}\right),\label{gluing:eq:parametre dans l'ordre}\end{equation}
and we can choose them in this particular order. 


\subsubsection{Control on the limiting measure}
In this section, the values of $\eta$ and $\epsilon$ and $n_0$ are fixed in such a way that the construction of the previous section holds. Note that everything in the section implicitly depends on those values.
On the event $\cE$, if we consider a fragment $f\in F_k$, we have a very good control on the values of $\pi_i(\cT(f))$ for $i\geq k$. Indeed set \[c_1=\prod_{k=0}^\infty (1-n_{k+1}^{-\frac{\epsilon}{4}})\quad \text{and} \quad c_2=\prod_{k=0}^\infty (1+n_{k+1}^{-\frac{\epsilon}{4}}).\]
Remark that both $c_1$ and $c_2$ are strictly positive real numbers. Using in cascade the estimations \eqref{gluing:eq:bounds mass grafted on fragments} which hold on the event $\cE$, we get that for $f\in F_k$ and $i\geq k$,
\begin{equation}\label{gluing:controle masse Tf}
\abs{\cT(f)\cap\cB_i}\leq c_2 \abs{f}\left(\prod_{j=k+1}^i a_j\right).
\end{equation}
In fact we can use the same argument for $\cB_k$, which is not empty on the event $\cE$. For $k$ large enough we can write
\begin{equation}\label{gluing:controle B_i}\abs{\cB_i}\in \abs{\cB_k}\left(\prod_{j=k+1}^i a_j\right)\cdot\intervalleff{c_1}{c_2}.\end{equation}
Remark that \eqref{gluing:controle B_i} combined with Lemma~\ref{gluing:lemme des ai} yields that almost surely on $\cE$,
\begin{equation}\label{gluing:controle masse Bk}
n_k^{\frac{\parametre(1-\beta)}{(\parametre-1)}+\petito{1}}\leq \abs{\cB_k}\leq n_k^{\frac{\parametre(1-\beta)}{(\parametre-1)}+\petito{1}},
\end{equation}
where the upper and lower-bound are both deterministic.
For $\pi_i$ the normalized mass measure on $\cB_i$, we have:
\[\pi_i(\cT(f))=\frac{\abs{\cT(f)\cap\cB_i}}{\abs{\cB_i}}\underset{\eqref{gluing:controle masse Tf}, \eqref{gluing:controle B_i}}{\leq} \frac{c_2}{c_1}\cdot \frac{\abs{f}}{\abs{\cB_k}}.\]
If $\pi$ is a sub-sequential limit of the $(\pi_k)$, using Portmanteau theorem (remark that $\pi$ is concentrated on the leaves and the leaves of $\cT(f)$ belong to the interior of $\cT(f)$), we get
\[\pi(\cT(f))\leq \frac{c_2}{c_1}\cdot\frac{\abs{f}}{\abs{\cB_k}}.\]
And then, \[\max_{f\in F_k}\pi(\cT(f))\leq \frac{c_2}{c_1} \frac{1}{\abs{\cB_k}}\max_{f\in F_k}\abs{f}.\]
We can now write, for all $r>0$, for all $x\in\cT$,
\begin{align*}\pi(\Ball (x,r)) \leq \sum_{f\in F_k, \ f\cap \Ball (x,r)\neq \emptyset} \pi(\cT(f))\leq \#\{f\in F_k, \ f\cap \Ball (x,r)\neq\emptyset\} \cdot \max_{f\in F_k}\pi(\cT(f)).
\end{align*}
Putting everything together we get 
\begin{equation}\label{gluing:majoration pi k}
\pi(\Ball (x,r))\leq  \#\{f\in F_k, \ f\cap \Ball (x,r)\neq\emptyset\} \cdot \frac{c_2}{c_1} \frac{1}{\abs{\cB_k}}\max_{f\in  F_k}\abs{f}.
\end{equation}

\subsubsection{Control on the number of fragments intersecting a ball}
From \eqref{gluing:majoration pi k}, we see that the last thing that we have to estimate is $\#\{f\in F_k, \ f\cap \Ball (x,r)\neq\emptyset\}$, the number of fragments of $F_k$ that have a non-empty intersection with a ball of radius $r$. Since the measure $\pi$ only charges $\bigcap_{k\geq 1}\cT(\cB_k)$, we are only interested in balls centred around points belonging to this set. Let us fix some notation again. For all $k\geq0$, we set 
\[\Delta_k:=\inf\enstq{\dist(x,y)}{x\in \cB_{k-1}, \ y \in \cB_{k}},\]
the set distance between $\cB_{k-1}$ and level $\cB_k$, for the integers $k$ for which it is possible. On the event $\cE$, this quantity is well-defined for all $k\geq1$ and the following upper and lower-bounds are almost surely satisfied
\begin{align}\label{gluing:eq:bounds Delta_k}n_k^{-\alpha +\petito1}= \frac{C^{-1}}{2} \min_{n\in G_{n_k}}\lambda_n \leq \Delta_k \leq C \max_{n\in G_{n_k}}\lambda_n  = n_k^{-\alpha +\petito1}.
\end{align}
Now let us state a lemma.
\begin{lemma}\label{gluing:nombre de fragments}Let $x\in\bigcap_{k\geq 1}\cT(\cB_k)$. For $k\geq0$, we denote $x_k:=[x]_{2n_k}\in\cB_k$. If $\bb_n$ is the block of $\cB_k$ such that $x_k\in \bb_n$, we have, $\forall r\in \intervalleff{0}{\Delta_k},$
	\begin{equation}\label{gluing:inclusion fragments}\enstq{f\in F_k}{f\cap \Ball (x,r)\neq\emptyset}\subset \enstq{f\in F_k}{\ f\cap \Ball (x_k,r)\cap \bb_n\neq \emptyset}.
	\end{equation}
\end{lemma}

The proof of this lemma is simple and left to the reader. It tells us is that in fact, if $r$ is small enough, then all the fragments $f\in F_k$ who intersect the ball of centre $x$ and radius $r$ belong to the same block. 
This will allow us in the sequel, combined with Lemma~\ref{gluing:lem:decomposition in fragments}\ref{gluing:it:nb fragments intersectant une boule}, to bound the number of fragments involved, which is what we wanted.

\subsubsection{Obtaining the lower-bound}
In order to get the lower-bound on the Hausdorff dimension of $\cL$ matching that of the theorem, we have to distinguish between the case $\alpha<\frac{1}{d}$ and the case $\alpha>\frac{1}{d}$. The case $\alpha=\frac{1}{d}$ can be recovered by a monotonicity argument, as seen in Section~\ref{gluing:monotonicity of Hausdorff dimension}. Whenever $d=0$, we have $\frac{1}{d}=+\infty$ and only the first case can happen.
\paragraph{Case $\beta>1$ and $\alpha<1/d$}
We use the construction of Section~\ref{gluing:construction of fragments} with $\eta=\alpha$. Recall \eqref{gluing:eq:parametre dans l'ordre} for the admissible parameters of the construction. In this case, if $\epsilon$ is fixed and small enough, the only condition on $\parametre$ implied by \eqref{gluing:eq:parametre dans l'ordre} is $\parametre>\frac{\beta-\alpha d}{1-\alpha d-\epsilon}$ since $\parametre>\frac{\alpha}{\eta}$ reduces to $\parametre>1$, which is already contained in the previous inequality because $\frac{\beta-\alpha d}{1-\alpha d-\epsilon}>1$. We define
\begin{equation*}
\Lambda_k:= \Delta_k \wedge \left(\frac{\min_{n\in G_{n_k}}\lambda_n}{\log n_k}\right).
\end{equation*}
Using \eqref{gluing:eq:bounds Delta_k}, we get $n_k^{-\alpha+\petito{1}}\leq \Lambda_k\leq n_k^{-\alpha+\petito{1}}$, almost surely, with deterministic bounds. Here the choice of $\frac{1}{\log n_k}$ is rather arbitrary and we could change it to any quantity that tends to $0$ and is $n_k^{\petito{1}}$ as $n_k\rightarrow\infty$. 
Now we claim the following
\begin{lemma}\label{gluing:lem:borne nombre fragments}
	On the event $\cE$, for all $k\geq 0,$ for any $d\geq 0$, we almost surely have
	\begin{equation}\label{gluing:eq:borne nombre fragments}\forall r\in \intervalleff{\Lambda_{k+1}}{\Lambda_k},  \quad \#\enstq{f\in F_k}{ f\cap \Ball (x,r)\neq\emptyset}\leq r^{d+\petito{1}} n_{k}^{\alpha \parametre d +\petito{1}}.
	\end{equation}
\end{lemma}
Note that the the bounds $\Lambda_k$ of the interval on which we consider $r$ are random, but the upper bound given by \eqref{gluing:eq:borne nombre fragments} is deterministic.
\begin{proof} 
	For $ r\in \intervalleff{\Lambda_{k+1}}{\Lambda_k},$ we have $r \leq \Lambda_k \leq \Delta_k$ so using Lemma~\ref{gluing:nombre de fragments}, with $x_k= \left[x\right]_{2n_k}$ and $n$ such that $x_k\in \bb_n$, we know that \eqref{gluing:inclusion fragments} holds.
	In the case $d>0$, the fragments of $F_k$ that come from $\bb_n$ were constructed as fragments of $\bB_n$ of size $\mathsf{r}_n:=\lambda_n^{-1}n_{k+1}^{-\alpha}$.
	Recall that we denote $F(\bB_n,\mathsf{r}_n)$ the set of these fragments, seen as subsets of $\bB_n$, and denote $\mathsf{x_k}$ the point of $\bB_n$ corresponding to $x_k\in\bb_n$. The analogue of the ball $\Ball(x_k,r)$ in $\bB_n$ is then the ball of centre $\mathsf{x_k}$ and radius $\mathsf{r'}_n:=\lambda_n^{-1} r$.
	From our definition of $\Lambda_k$ we have 
	\[\mathsf{r'}_n= \lambda_n^{-1} r\leq \lambda_n^{-1}\frac{\min_{n\in G_{n_k}}\lambda_n}{\log n_k}\leq \frac{1}{\log n_k} \tend{k\rightarrow\infty}{} 0,\] as well as $\mathsf{r}_n:=\lambda_n^{-1}n_{k+1}^{-\alpha}=n_k^{\alpha-\parametre \alpha +\petito{1}}\rightarrow 0$ when $k\rightarrow \infty$. 
	Applying Lemma~\ref{gluing:lem:decomposition in fragments}\ref{gluing:it:nb fragments intersectant une boule} yields
	\begin{align*}\#\enstq{f\in F(\bB_n,\mathsf{r}_n)}{\ f\cap \Ball (\mathsf{x_k},\mathsf{r}_n)\neq \emptyset} &\underset{\text{Lem.} \ref{gluing:lem:decomposition in fragments}\ref{gluing:it:nb fragments intersectant une boule}}{\leq} (\mathsf{r}_n\vee \mathsf{r'}_n)^{d+\psi(\mathsf{r}_n\vee \mathsf{r'}_n)}\cdot \mathsf{r}_n^{-d+\phi(\mathsf{r}_n)}\\
	&\leq \left((\lambda_n^{-1}n_{k+1}^{-\alpha})\vee (\lambda_n^{-1} r)\right)^{d+\petito{1}}\cdot (\lambda_n^{-1}n_{k+1}^{-\alpha})^{-d+\petito{1}}\\
	&\leq r^{d+\petito{1}} n_{k}^{\alpha \parametre d +\petito{1}},
	\end{align*}
	and the last quantity is deterministic. Since any fragment in $\enstq{f\in F_k}{f\cap \Ball (x,r)\neq\emptyset}$ corresponds to a fragment in $\enstq{f\in F(\bB_n,\mathsf{r}_n)}{\ f\cap \Ball (\mathsf{x_k},\mathsf{r}_n)\neq \emptyset}$, the cardinal of $\enstq{f\in F_k}{ f\cap \Ball (x,r)\neq\emptyset}$ is almost surely bounded above by the last display, which proves that \eqref{gluing:eq:borne nombre fragments} holds whenever $d>0$. 
	In the case $d=0$, from our definition of fragments and the property \eqref{gluing:property P0}, we easily have
	\[\#\enstq{f\in F_k}{ f\cap \Ball (x_k,r)\cap \bb_n\neq \emptyset}\leq C\leq r^{d+\petito{1}} n_{k}^{\alpha \parametre d +\petito{1}},\]
	and so  \eqref{gluing:eq:borne nombre fragments} also holds whenever $d=0$.
\end{proof}
We can compute:
\begin{align*}
\pi(\Ball (x,r))&\underset{\eqref{gluing:majoration pi k}}{\leq}  \#\{f\in F_k, \ f\cap \Ball (x,r)\neq\emptyset\} \cdot \frac{c_2}{c_1}\cdot \frac{1}{\abs{\cB_k}} \cdot \max_{f\in F_k}\abs{f}\\
&\underset{\eqref{gluing:eq:borne nombre fragments},\eqref{gluing:controle masse Bk},\eqref{gluing:max poids fragment}}{\leq}  (r^{d+\petito{1}} n_{k}^{\alpha \parametre d +\petito{1}} )\cdot r^{\petito{1}} \cdot (n_k^{\frac{\parametre(\beta-1)}{(\parametre-1)}+\petito{1}})\cdot (n_{k}^{-\alpha \parametre d +\alpha d-\beta+\petito{1}})\\
&\leq r^{d+\petito{1}} \cdot n_k^{\frac{1}{(\parametre-1)}(\parametre\alpha d -\alpha d+\beta-\parametre+\petito{1})}\\
&\leq  r^{d-\frac{1}{\alpha\parametre(\parametre-1)}(\parametre\alpha d -\alpha d+\beta-\parametre)+\petito{1}}.
\end{align*}
In the last line we used that $r>\Lambda_{k+1}\geq n_{k+1}^{-\alpha+\petito{1}}=n_k^{-\parametre \alpha+\petito{1}}$ and so $n_k>r^{-\frac{1}{\alpha\parametre}+\petito{1}}$ and the fact that $\parametre\alpha d -\alpha d+\beta-\parametre<0$ because $\parametre>\frac{\beta-\alpha d}{1-\alpha d}$.
Let us now maximise the quantity $d-\frac{\parametre\alpha d -\alpha d+\beta-\parametre}{\alpha\parametre(\parametre-1)}$ for $\parametre\in\intervalleoo{\frac{\beta-\alpha d}{1-\alpha d}}{+\infty}$.
It is an easy exercise to see that the maximum is attained at $\bar{\parametre}=\frac{\beta-\alpha d+\sqrt{(\beta-1)(\beta-\alpha d)}}{1-\alpha d}$, with value \[\frac{2\beta-1-2\sqrt{(\beta-1)(\beta-\alpha d)}}{\alpha}.\]
If we fix $\epsilon$ small enough then, the value of $\parametre$ that maximises the last display satisfies $\parametre>\frac{\beta-\alpha d}{1-\alpha d-\epsilon}$ and so, using this value to construct $\pi$, we get that on the event $\cE$, for all $x\in\bigcap_{k\geq 1}\cT(\cB_k)$, 
\[\pi(\Ball (x,r))\leq r^{\frac{2\beta-1-2\sqrt{(\beta-1)(\beta-\alpha d)}}{\alpha}+\petito{1}},\]
which allows us to conclude using Lemma~\ref{gluing:frostman lemma}.
\paragraph{Case $\beta>1$ and $\alpha>1/d$}
Here we suppose that $d>0$. We fix $\eta<\frac{1}{d}$ which we suppose to be very close to $\frac{1}{d}$ and a small $\epsilon>0$ and use the construction of Section~\ref{gluing:construction of fragments} with these values, which satisfy \eqref{gluing:eq:parametre dans l'ordre} if we take $\parametre>\max\left(\frac{\beta-\alpha d}{1-\eta d-\epsilon},\frac{\alpha}{\eta}\right)$.

$\bullet$ For $r\in\intervalleff{\Lambda_{k+1}}{n_{k+1}^{-\eta}}$, we apply Lemma~\ref{gluing:lem:borne nombre fragments} to get 
\[\#\{f\in F_k, \ f\cap \Ball (x,r)\neq\emptyset\}\leq n_k^{\petito{1}}.\]
Now we can use the upper-bound \eqref{gluing:majoration pi k}, replacing term by term
\begin{align*}
\pi(\Ball (x,r))&\underset{\eqref{gluing:majoration pi k}}{\leq}  \#\{f\in F_k, \ f\cap \Ball (x,r)\neq\emptyset\} \cdot \frac{c_2}{c_1} \frac{1}{\abs{\cB_k}}\max_{f\in  F_k}\abs{f}.\\
&\underset{\eqref{gluing:controle masse Bk},\eqref{gluing:max poids fragment}}{\leq} n_k^{\petito{1}}\cdot \frac{c_2}{c_1} n_k^{\frac{\parametre(\beta-1)}{(\parametre-1)}+\petito{1}} n_{k}^{-\eta d \parametre+\alpha d-\beta+\petito{1}}\\
&\leq n_k^{\frac{\parametre(\beta-1)}{(\parametre-1)}-\parametre\eta d+\alpha d-\beta+\petito{1}}.
\end{align*}
Hence, using the fact that $r>\Lambda_{k+1}=n_k^{-\parametre \alpha+\petito{1}}$ and that the exponent in the last display is negative, we get 
\begin{align*}
\pi(\Ball (x,r))&\leq r^{-\frac{1}{\parametre\alpha}\cdot(\frac{\parametre(\beta-1)}{(\parametre-1)}-\parametre\eta d +\alpha d-\beta)+\petito{1}}\\
&\leq r^{\frac{\eta d}{\alpha}-\frac{1}{\alpha\parametre}(\frac{\parametre(\beta-1)}{(\parametre-1)}+\alpha d-\beta)+\petito{1}}.
\end{align*}
$\bullet$ For $r\in\intervalleff{n_{k+1}^{-\eta}}{\Lambda_k}$, we have once again using Lemma~\ref{gluing:lem:borne nombre fragments}, 
\[\#\{f\in F_k, \ f\cap \Ball (x,r)\neq\emptyset\}\leq r^{d+\petito{1}} n_{k+1}^{\eta d +\petito{1}}.\]
Replacing in \eqref{gluing:majoration pi k} yields
\begin{align*}
\pi(\Ball (x,r))&\leq  r^{d+\petito{1}} n_{k+1}^{\eta d +\petito{1}} \cdot \frac{c_2}{c_1} \frac{1}{n_k^{\frac{\parametre(1-\beta)}{(\parametre-1)}+\petito{1}}} n_{k}^{-\eta d \parametre+\alpha d-\beta+\petito{1}}\\
&\leq  r^{d+\petito{1}} \cdot n_k^{\alpha d-1+\frac{\beta-1}{\parametre-1}+\petito{1}}.
\end{align*}
Since $r\leq \Lambda_k \leq n_k^{-\alpha+\petito{1}}$, we have $n_k \leq r^{-\frac{1}{\alpha}+\petito{1}}$. Since the quantity $\alpha d-1+\frac{\beta-1}{\parametre-1}$ is positive, we can write 
\begin{align*}
\pi(\Ball (x,r))&\leq r^{d+\petito{1}} \cdot r^{-\frac{1}{\alpha}(\alpha d-1+\frac{\beta-1}{\parametre-1})+\petito{1}}\\
&\leq r^{d-\frac{\alpha d-1}{\alpha}-\frac{\beta-1}{\alpha(\parametre-1)}+\petito{1}}\\
&\leq r^{\frac{1}{\alpha}-\frac{\beta-1}{\alpha(\parametre-1)}+\petito{1}}.
\end{align*}
Now the result is obtained by taking $\epsilon\rightarrow 0$ and $\eta \rightarrow \frac{1}{d}$ and $\parametre\rightarrow\infty$.

To conclude the proof of Proposition~\ref{gluing:prop:minoration dimension hausdorff}, we have to prove that in the case $\alpha=1/d$, the dimension of $\cL$ is bounded below by $d$. To that end, we use the monotonicity of the Hausdorff dimension of $\cL$, with respect to the scaling factors $(\lambda_n)$, proved in Section~\ref{gluing:monotonicity of Hausdorff dimension}.
Suppose the sequences $(\lambda_n)$ and $(w_n)$ satisfy Hypothesis~\ref{gluing:condition beta grand} for $\alpha=1/d$. If for some $\epsilon>0$, we set for all $n\geq1$,  $\lambda'_n=n^{-\epsilon}\lambda_n$, then the sequences $(\lambda'_n)$ and $(w_n)$ satisfy Hypothesis~\hyperref[condition beta grand]{$\diamond_{\alpha+\epsilon,\beta}$}. Now for $n\geq 1$, we have $\lambda_n\geq \lambda'_n$, and  $\cT$ is compact with probability $1$ from Proposition~\ref{gluing:prop:compacité+majoration dimension}. Hence \eqref{gluing:monotonicity} holds and so we have, a.s.\
\[\dim(\cL)\geq \frac{1}{\alpha+\epsilon}.\]
In the end, $\dim(\cL)\geq d$.

\begin{proof}[Proof of Theorem~\ref{gluing:theoreme principal}]
	Use Proposition~\ref{gluing:prop:compacité+majoration dimension}, Proposition~\ref{gluing:dimension} for the upper-bounds and Proposition~\ref{gluing:prop:minoration beta petit} and Proposition~\ref{gluing:prop:minoration dimension hausdorff} for the lower-bounds.
\end{proof}

\appendix
	\section{Appendix}
	\subsection{Lifting to the Urysohn space}\label{gluing:app:embedded construction}
	In this section, we prove that it is always possible to work with random measured metric spaces that are embedded in the Urysohn space.
	Let us first recall the definition of the Gromov-Hausdorff-Prokhorov distance. If $(X,d)$ is a metric space, and $A\subset X$ then we denote
	$A^{(\epsilon)}:=\enstq{x\in X}{d(x,A)<\epsilon},$
	the $\epsilon$-fattening of $A$. Then $\mathrm{d_H}$ the Hausdorff distance on the set of non-empty compact subsets of $X$, is defined as \[\mathrm{d_H}(K,K'):=\inf\enstq{\epsilon>0}{K\subset (K')^\epsilon, \ K'\subset (K)^\epsilon}.\] Also we denote the so-called Lévy-Prokhorov distance on the Borel probability measures by
	\[\mathrm{d_{LP}}(\nu,\nu'):=\inf\enstq{\epsilon>0}{\forall A \in \cB(X), \ \nu(A)\leq \nu'\left((A)^\epsilon\right)+\epsilon \ \text{and} \ \nu'(A)\leq \nu\left((A)^\epsilon\right)+\epsilon},\]
	where $\cB(X)$ is the set of Borel sets of $X$.
	Now let $(X,\dist,\rho,\nu)$ and $(X',\dist',\rho',\nu')$ be two compact, rooted, metric spaces endowed with a probability measure. Their Gromov-Hausdorff-Prokhorov distance is defined as
	\[\mathrm{d_{GHP}}\left((X,\dist,\rho,\nu),(X',\dist',\rho',\nu')\right):= \inf_{E,\phi, \phi'}\max(\dist(\rho,\rho'),\mathrm{d_H}(\phi(X),\phi(X')),\mathrm{d_{LP}}(\phi_* \nu, \phi'_* \nu')),\]
	where the infimum is taken over all Polish spaces $(E,\delta)$ and all isometric embeddings $\phi: \ X\rightarrow E$ and $\phi': \ X'\rightarrow E$, of respectively $X$ and $X'$ into $E$. The notation $\phi_*\nu$ denotes the push-forward of the measure $\mu$ through the map $\phi$. As it is, this is only a pseudo-distance and becomes a distance on the set $\K$ of GHP-isometry (root and measure preserving isometry) classes of compact, rooted, metric spaces endowed with a probability measure, which from \cite[Theorem~2.5]{abraham_note_2013}, is a Polish space. We consider all our blocks as (possibly random) elements of the set $\K$. 
	
	We would like to see all the blocks as compact subsets of the same space. To that end, we consider $(U,\delta)$ the Urysohn space, and fix a point $u_0\in U$. The space $U$ is defined as the only Polish metric space (up to isometry) which has the following extension property (see \cite{husek_urysohn_2008} for constructions and basic properties of $U$): given any finite metric space $X$, and any point $x\in X$, any isometry from $X\setminus\{x\}$ to $U$ can be extended to an isometry from $X$ to $U$.
	This property ensures in particular that any separable metric space can be isometrically embedded into $U$. In what follows we will use the fact that if $(K,\dist,\rho)$ is a rooted compact metric space, there exists an isometric embedding of $K$ to $U$ such that $\rho$ is mapped to $u_0$. We set
	\[\K(U):=\enstq{(K,\nu)}{K\subset U, \ K \text{ compact}, u_0\in K, \ \nu \  \text{is a Borel measure and }\supp(\nu)\subset K},\]
	where $\supp(\nu)$ denotes the topological support of $\nu$. We endow $\K(U)$ with the "Hausdorff-Prokhorov" distance
	\[\mathrm{d_{HP}}((K,\nu),(K',\nu'))=\max\left(\mathrm{d_H}(K,K'),\mathrm{d_{LP}}(\nu,\nu')\right).\]
	It is easy to see that $\left(\K(U),\mathrm{d_{HP}}\right)$ is a Polish space. Now, we have a map $f:\K(U)\rightarrow\K$, which maps every $(K,\nu)$ to the isometry class of $(K,\delta_{\vert_K},u_0,\nu)$ in $\K$. This map is continuous and hence measurable. The properties of $U$ ensure that $f$ is surjective. Using a theorem of measure theory from \cite{lubin_extensions_1974}, every probability distribution $\tau$ on $\K$ can be lifted to a probability measure $\sigma$ on $\K(U)$, such that $f_*\sigma=\tau$. 
	Hence, for all $n\geq 1$, we can have a version of $(\bb_n,\bd_n,\brho_n,\bnu_n)=(\bB_n,\lambda_n\cdot \bD_n,\bRho_n,w_n\cdot \bNu_n)$ that is embedded in the space $U$.
	\subsection{Hausdorff dimension}\label{gluing:subsec:hausdorff dimension}
	
	We recall some notations and definitions that are in relation with Hausdorff dimension and that we use throughout the paper. Let $(X,\dist)$ be a metric space and $\delta>0$. We say that the family $(O_i)_{i\in I}$ of subsets of $X$ is a $\delta$-cover of $X$ if it is a covering of $X$, and the set $I$ is at most countable and for all $i\in I$, the set $O_i$ is such that its diameter satisfies $\diam(O_i)<\delta$. We set
	\[\cH_s^\delta(X):=\inf\enstq{\sum_{i\in I}\diam(O_i)^s}{(O_i)_{i\in I} \text{ is a } \delta\text{-cover of }X},\]
	As this quantity increases when $\delta$ decreases to $0$, we define its limit
	\[\cH_s(X):=\lim_{\delta\rightarrow 0}\cH_s^\delta(X)\in \intervalleff{0}{\infty},\]
	the \emph{$s$-dimensional Hausdorff measure} of $X$. Now the Hausdorff dimension of $X$ is defined as 
	\[\mathrm{dim_H}(X):=\inf\enstq{s>0}{\cH_s(X)=0}=\sup\enstq{s>0}{\cH_s(X)=\infty}.\]
	We refer to \cite{falconer_fractal_2014} for details.
	A useful tool for deriving lower-bounds on the Hausdorff dimension of a metric space is the so called Frostman's lemma. In this paper we use the following version.
	\begin{lemma}[Frostman's lemma]\label{gluing:frostman lemma}
		Let $(X,\dist)$ be a metric space. If there exists a non-zero finite Borel measure $\mu$ on $X$ and $s>0$ such that for $\mu$-almost every $x\in X$, we have
		\[\mu\left(\Ball (x,r)\right)\underset{r \rightarrow 0}{\leq} r^{s+\petito{1}},\]
		then \[\mathrm{dim_H}(X)\geq s.\]
	\end{lemma}
	
	\subsection{Decomposition into \emph{fragments}}\label{gluing:subsec:decomposition in fragments}
	In this section, we prove Lemma~\ref{gluing:lem:decomposition in fragments}. We first construct our fragments in a deterministic setting and then show how we can apply this to random blocks.
	\paragraph{Decomposition of a deterministic block}
	Let $(\bb,\bd,\brho,\bnu)$ be a (deterministic) pointed compact metric space endowed with a Borel probability measure. We are interested in how we can decompose $\bb$ into a partition of subsets that all have approximately the same diameter $r$. 
	For $r>0$, we set 
	
	\[\cP_r(\bb):=\enstq{\{x_1,x_2, \dots, x_n\}\subset \bb}{ n\geq1 \text{ and }\ \forall i \neq j, \ \bd(x_i,x_j)\geq \frac{r}{2}}.
	\]
	It is easy to verify that we can find $p= \{x_1,x_2,\dots,x_n\}\in\cP_r(\bb)$ such that $\bb\subset\bigcup_{i=1}^n \Ball (x_i,r)$ and the balls $\left(\Ball (x_i,\frac{r}{4})\right)_{1\leq i\leq n}$ are disjoint. Indeed, any $\frac{r}{2}$-net of $\bb$ belongs to $\cP_r(\bb)$ (they are the maximal elements of $\cP_r(\bb)$ for the order relation of inclusion). 
	We denote $\cP_r^*(\bb)$ the set 
	\[\cP_r^*(\bb):=\enstq{\{x_1,x_2, \dots, x_n\}\in\cP_r(\bb)}{\bb\subset\bigcup_{i=1}^n \Ball (x_i,r)},\]
	which is non-empty from what precedes. Considering the collection of balls of radius $r$ with centres in $p\in\cP_r^*(\bb)$ gives rise to a covering of $\bb$ which is close to optimal in a sense specified by the following lemma. We recall the notation $N_r(\bb)$ which denotes the minimal number of balls of radius $r$ needed to cover $\bb$.
	\begin{lemma}\label{gluing:majoration M_r}
		For any $p= \{x_1,x_2,\dots,x_n\}\in\cP_r^*(\bb)$, we have
		$n\leq N_{r/4}(\bb)$.
	\end{lemma}
	\begin{proof}
		Let $p= \{x_1,x_2,\dots,x_n\}\in\cP_r^*(\bb)$. Remark that, for any set $S$ such that the union of the balls $\left(\Ball(s,\frac{r}{4})\right)_{s\in S}$ covers $\bb$, each of the balls $\Ball(x_i,r/4)$, for $1\leq i\leq n$, contains at least a point of $S$. Since those balls are disjoint, the cardinality of $S$ is at least $n$.
	\end{proof}
	
	From any element $p= \{x_1,x_2,\dots,x_n\}\in\cP_r^*(\bb)$, we can then construct a partition of $\bb$, into subsets $(f_i)_{1\leq i \leq n}$ that we call fragments, and such that \[\forall i\in\intervalleentier{1}{n}, \quad \Ball \left(x_i,\frac{r}{4}\right)\subset f_i\subset \Ball (x_i,r).\]
	We define the $(f_i)$ recursively as follows,
	\[\left\lbrace\begin{aligned}
	f_1&:=\enstq{x\in\bb}{\bd(x,x_1)=\min_{1\leq i\leq n} \bd(x,x_i)}\\
	f_{k+1}&:= \enstq{x\in\left(\bb\setminus \bigcup_{i=1}^k f_i\right)}{\bd(x,x_{k+1})=\min_{1\leq i\leq n} \bd(x,x_i)}.\\
	\end{aligned}\right. \]
	If we suppose that $\bb$ satisfies the condition \eqref{gluing:controle boules}, and that $r< r_0$ then, for $p= \{x_1,x_2,\dots,x_n\}\in\cP_r^*(\bb)$, we get
	\[n \left(\frac{r}{4}\right)^{d+\varphi(r)}\leq \sum_{i=1}^n \bnu\left(\Ball\left(x_i,\frac{r}{4}\right)\right)\leq \bnu(\bb)=1,\]
	so that we have 
	\begin{equation}\label{gluing:nombre fragments r0}
	n\leq \left(\frac{r}{4}\right)^{-d-\varphi(r/4)},
	\end{equation} 
	and also, for all $i\in\intervalleentier{1}{n}$,
	\begin{equation}\label{gluing:eq:diam et masse des fragments}
	\diam f_i \leq 2r \qquad \text{and} \qquad \left(\frac{r}{4}\right)^{d+\varphi(r/4)} \leq \bnu(f_i) \leq r^{d-\varphi(r)}.
	\end{equation}
	Let us state another lemma.
	
	\begin{lemma}\label{gluing:boules et fragments}
		Let $r_0<1$. Under the condition \eqref{gluing:controle boules}, there exists two functions $\psi$ and $\phi$ defined on $\intervalleff{0}{r_0/3}$, which tend to $0$ at $0$ such that for all $ r\in\intervalleoo{0}{r_0/3}$, for all $p=\{x_1,x_2,\dots,x_n\}\in \cP_r^*(\bb)$ and fragments $(f_i)$ constructed as above, we have 
		\[ \forall x\in\bb, \forall r'\in \intervalleoo{0}{r_1}, \quad \#\enstq{1\leq i \leq n}{\Ball (x,r')\cap f_i\neq\emptyset} \leq (r\vee r')^{d+\psi(r\vee r')}\cdot r^{-d+\phi(r)}.\]
	\end{lemma}
	\begin{proof}
		Let $x\in\bb$. If for an $i\in\intervalleentier{1}{n}$, we have $y\in \Ball (x,r')\cap f_i\neq\emptyset$, then $\bd(x,y)<r'$ and $\bd(x_i,y)<r$ so $\bd(x,x_i)<r+r'$, and so we get that $\Ball (x_i,r)\subset \Ball (x,r'+2r)$. Then, using that $f_i\subset \Ball (x_i,r)$, 
		\[
		\left(\bigcup_{i: \ f_i\cap \Ball (x,r')\neq\emptyset}\Ball \left(x_i,\frac{r}{4}\right)\right) \subset \left(\bigcup_{i: \ f_i\cap \Ball (x,r')\neq\emptyset}f_i \right)\subset \Ball (x,r'+2r).
		\]
		We can use the condition \eqref{gluing:controle boules} to get, for all $r,\ r'\in\intervalleff{0}{r_0/3},$
		\[\#\enstq{1\leq i \leq n}{\Ball (x,r')\cap f_i\neq\emptyset}\cdot \left(\frac{r}{4}\right)^{d+\varphi(r/4)}\leq (r'+2r)^{d-\varphi(r'+2r)}.\]
		And so,
		\begin{align*}\#\enstq{1\leq i \leq n}{\Ball (x,r')\cap f_i\neq\emptyset}& \leq 4^{d+\varphi(r/4)} \frac{(r'+2r)^{d-\varphi(r'+2r)}}{r^{d+\varphi(r/4)}}\\
		&\leq 4^{d+\varphi(r/4)}\frac{(3 (r\vee r'))^{d-\varphi(r'+2r)}}{r^{d+\varphi(r/4)}}\\
		&\leq 4^{d+\varphi(r/4)} 3^{d-\varphi(r'+2r)}(r\vee r')^{d-\varphi(r'+2r)}r^{-d-\varphi(r/4)},\\
		&\leq (r\vee r')^{d-\varphi(3(r\vee r'))}r^{-d-\varphi(r/4)+\frac{\log(12^{3d/2})}{\log r}},
		\end{align*}
		which proves the lemma.
	\end{proof}
	This lemma gives us an abstract result for the existence and the properties of these decompositions in fragments. The next paragraph explains a procedure to construct one using a sequence of i.i.d. random points, on a possibly random block.
	
	\paragraph{Finding an element of $\cP_r^*(\bb)$}\label{gluing:random fragments}
	Suppose that the measure $\bnu$ charges all open sets. Remark that this is almost surely true for our random block $(\bB,\bD,\bRho,\bNu)$ because they satisfy Hypothesis~\ref{gluing:hypothese d}\ref{gluing:dimension d}. Let $(X_n)_{n\geq 1}$ be a sequence of i.i.d.\ random variables with law $\bnu$. Let us construct a random element of $\cP_r^*(\bb)$ for some fixed $r>0$. Define the set $E_n$ recursively as follows:
	\[E_1:=\{1\} \quad \text{and} \quad \left\lbrace\begin{aligned}
	E_{n+1}&:=E_n \quad \text{if} \quad  X_{n+1} \in \bigcup_{i\in E_n}\Ball\left(X_i,\frac{r}{2}\right),\\
	E_{n+1}&:= E_n\cup \{n+1\} \quad \text{otherwise}.\\
	\end{aligned}\right.\]
	We set $E_\infty=\bigcup_{n\geq 1} E_n$. Note that from the construction, $\{X_i,\ i\in E_\infty\}\in \cP_{r}(\bb)$.
	\begin{lemma}\label{gluing:lem:random element of prb}
		Almost surely, we have
		$\{X_i,\ i\in E_\infty\}\in \cP_r^*(\bb)$.
	\end{lemma}
	\begin{proof}
		The fact that the balls $\left(\Ball (X_i,\frac{r}{4})\right)_{i\in E_\infty}$ are disjoint follows directly from the construction. Now let $x\in\bb$. Since $\bnu\left(\Ball\left(x,\frac{r}{4}\right)\right)>0$, by the the Borel-Cantelli lemma there exists at least one $n$ such that $X_n\in\Ball\left(x,\frac{r}{4}\right)$. If $n\in E_\infty$ then $x\in\Ball\left(X_n,\frac{r}{4}\right)$.
		Otherwise $n\notin E_\infty$, in which case there exists $k\leq n$ such that $X_n\in \Ball \left(X_k,\frac{r}{2}\right)$ and so $x\in \Ball \left(X_k,\frac{3}{4}r\right)$. In both cases
		\[\Ball\left(x,\frac{r}{4}\right)\subset\bigcup_{i\in E_\infty} \Ball (X_i,r).\]
		Since we can apply the same reasoning to every point of a countable dense sequence $(y_k)_{k\geq 1}$, the lemma is proved.
	\end{proof}
	\begin{proof}[Proof of Lemma~\ref{gluing:lem:decomposition in fragments}]
		This is just a consequence of Lemma~\ref{gluing:majoration M_r}, Lemma~\ref{gluing:boules et fragments} and Lemma~\ref{gluing:lem:random element of prb} and equations \eqref{gluing:nombre fragments r0} and  \eqref{gluing:eq:diam et masse des fragments}, which almost surely apply to the random block $(\bB,\bD,\bRho,\bNu)$. 
	\end{proof}

	\subsection{Computations}\label{gluing:subsec:computations}
	\begin{lemma}\label{gluing:lem:croissance logarithmique}
		Suppose that there exists $\gamma\geq 0$ such that for all $n\in\N$, $W_n\leq n^\gamma$. Then there exists a constant $C$ such that
		\begin{align*}
		\sum_{k=1}^{n}\frac{w_k}{W_k}\leq C\log n.
		\end{align*}
	\end{lemma}
	\begin{proof}
		If the series $\sum w_k$ converges then the results is trivial so let us suppose that it diverges. For $k\geq0$, we define $n_k:=\inf\enstq{i\geq 1}{W_i\geq 2^k}$ and write
		\begin{align*}
		\sum_{k=n_0}^{n}\frac{w_k}{W_k} \leq \sum_{i=0}^{\left\lceil\frac{\log W_n}{\log 2}\right\rceil} \sum_{k=n_i}^{n_{i+1}-1} \frac{w_k}{W_k}\leq \sum_{i=0}^{\left\lceil\frac{\log W_n}{\log 2}\right\rceil} \frac{1}{2^i}\sum_{k=n_i}^{n_{i+1}-1} w_k&\leq \sum_{i=0}^{\left\lceil\frac{\log W_n}{\log 2}\right\rceil} \frac{2^{i+1}}{2^i}\\
		&\leq 2 \left\lceil\frac{\log W_n}{\log 2}\right\rceil,
		\end{align*}
		which grows at most logarithmically thanks to our assumption on the sequence $(W_n)$.
	\end{proof}
	
	\begin{lemma}\label{gluing:divergence en log}
		Let $\beta<1$ and assume that $w_n\leq n^{-\beta+\petito{1}}$ and $W_n=n^{1-\beta+\petito{1}}$ and that for some $\epsilon>0$ we have:
		\[\liminf_{n\rightarrow\infty}\frac{1}{W_n}\underset{k\in G^\epsilon}{\sum_{k=1}^{n}}w_k>0.\]
		Then there exists a constant $C_\epsilon$ such that for $N$ large enough we have
		\[\underset{k\in G^\epsilon}{\sum_{k=N}^{N^{1+\epsilon}}}\frac{w_k}{W_k}\geq C_\epsilon \log N.\]
	\end{lemma}
	\begin{proof}
		Let $c$ be such that, for $n$ large enough $\frac{1}{W_n}\sum_{k=1}^{n}w_k\ind{k\in G^\epsilon}>c$. Let $C:=\frac{3}{c}$, note that $C>1$ because $c\leq 1$. For all $i\geq1$, we set $n_i=\inf\enstq{n}{W_n\geq C^i}$. For all $i\geq1$, we have
		$W_{n_i-1}\leq C^i\leq W_{n_i}\leq C^i + w_{n_i}$.
		We get, 
		\begin{align*}
		\sum_{k=n_i+1}^{n_{i+1}}w_k\ind{k\in G^\epsilon}\geq c W_{n_{i+1}} - W_{n_i}\geq cC^{i+1}-C^i-w_{n_i} &\geq C^i(cC-1-\frac{w_{n_i}}{C^i})\\
		&\geq C^i(1+\petito{1}).
		\end{align*}
		for $i$ tending to infinity.
		Now for $N$ a large integer, we set 
		\[I_N:= \inf\enstq{i}{n_i\geq N}= \left\lceil\frac{\log W_N}{\log C}\right\rceil \quad \text{and} \quad J_N:= \sup\enstq{i}{n_i\leq N^{1+\epsilon}}= \left\lfloor\frac{\log W_{\lfloor N^{1+\epsilon}\rfloor}}{\log C}\right\rfloor.\]
		Then we compute
		\begin{align*}
		\sum_{k=N}^{N^{1+\epsilon}}\frac{w_k}{W_k}\ind{k\in G^\epsilon}&\geq \sum_{i=I_N}^{J_N}\sum_{k=n_i+1}^{n_{i+1}}\frac{w_k}{W_k}\ind{k\in G^\epsilon}\\
		&\geq \sum_{i=I_N}^{J_N}\frac{1}{W_{n_{i+1}}}\sum_{k=n_i+1}^{n_{i+1}}w_k\ind{k\in G^\epsilon}\\
		&\geq \sum_{i=I_N}^{J_N} \frac{1}{C^{i+1}(1+\petito{1})} C^i(1+\petito{1})\\
		&\geq\frac{J_N-I_N}{C} (1+\petito{1}).
		\end{align*}
		We finish the proof by noting that, thanks to the hypothesis on the growth of $W_n$, the last display grows logarithmically in $N$.
	\end{proof}
	
	\begin{proof}[Proof of Lemma~\ref{gluing:lemme des ai}]
		From our assumptions, it is easy to see that we have $a_k=n_k^{1-\beta+\petito{1}}$. For all $k$, we write
		\[\log a_k = (1-\beta + r_k)\log n_k,\]
		with $r_k\rightarrow0$ as $k\rightarrow\infty$. We write
		\begin{align}\label{gluing:eq:sum log ai}
		\sum_{i=1}^k \log a_i &= \sum_{i=1}^k (1-\beta + r_i)\log n_i=\sum_{i=0}^{k-1} (1-\beta + r_{k-i})\log n_{k-i}.
		\end{align}
		For any $k\geq0$, from the recursive definition of the sequence $(n_k)$, we have $n_{k+1}-1 < n_k^\gamma \leq n_{k+1}$, which entails $\log n_k=\frac{1}{\gamma}\log n_{k+1}+s_{k}$, with $\abs{s_{k}}\leq 1$. Using this recursively yields
		\begin{align*}
		\abs{\log n_{k-i}-\frac{1}{\gamma^i}\log n_k}\leq \frac{\gamma}{1-\gamma}. 
		\end{align*}
		Hence using \eqref{gluing:eq:sum log ai} and the fact that $\log n_k$ grows exponentially in $k$,
		\begin{align*}
		\sum_{i=1}^k \log a_i &= \log n_k \left((1-\beta)\sum_{i=0}^{k-1}\frac{1}{\gamma^i} + \underset{\rightarrow 0}{\underbrace{\sum_{i=0}^{k} \frac{r_{k-i}}{\gamma^i}}} \right)+ \underset{=\grandO{k}}{\underbrace{\sum_{i=0}^{k-1} (1-\beta + r_{k-i}) \left(\log n_{k-i}-\frac{1}{\gamma^i}\log n_k\right)}}\\
		&=\log n_k\left(\frac{(1-\beta)\gamma}{\gamma-1}+\petito{1}\right),
		\end{align*}
		which proves the lemma.
	\end{proof}
\printbibliography
\end{document}